\documentclass[11pt]{amsart}
\usepackage[utf8]{inputenc}
\usepackage[pdftex]{graphicx}
\usepackage{amssymb}
\usepackage{amsmath}
\usepackage{amsfonts}
\usepackage{amsthm}
\usepackage{bm}
\usepackage{mathrsfs}
\usepackage{appendix}
\usepackage{enumerate}
\usepackage[left=3cm, right=3cm, bottom=3.4cm]{geometry}
\usepackage[stretch=30,shrink=30]{microtype}
\usepackage{stmaryrd}
\usepackage[normalem]{ulem} 
\usepackage{mathtools}
\usepackage{todonotes}
\usepackage{pgf,tikz}
\usetikzlibrary{decorations.markings}
\usetikzlibrary{arrows}

\usepackage{xcolor}
\definecolor{antiquefuchsia}{rgb}{0.57, 0.36, 0.51}
\definecolor{azure}{rgb}{0.0, 0.5, 1.0}

\usepackage[colorlinks = true, linkcolor = black, citecolor = antiquefuchsia]{hyperref}

\makeatother

\numberwithin{equation}{section}

\newtheorem{theorem}{Theorem}[section]
\newtheorem{lemma}[theorem]{Lemma}
\newtheorem{proposition}[theorem]{Proposition}

\newtheorem{open}[theorem]{Problem}
\newtheorem{definition}[theorem]{Definition}
\theoremstyle{remark}

\newcommand{\N}{\mathbb{N}}
\newcommand{\Z}{\mathbb{Z}}
\newcommand{\R}{\mathbb{R}}

\newcommand{\T}{{\mathbf T}}

\DeclareMathOperator{\Per}{Per}

\newcommand{\HH}{{\mathcal H}}

\def\Z{\ensuremath\mathbb{Z}}

\usepackage[capitalise,nameinlink]{cleveref}
\crefformat{equation}{#2(#1)#3}
\crefrangeformat{equation}{#3(#1)#4 to~#5(#2)#6}
\crefmultiformat{equation}{#2(#1)#3}{ and~#2(#1)#3}{, #2(#1)#3}{ and~#2(#1)#3}

\newcommand{\abs}[1]{\left\lvert #1 \right\rvert}
\newcommand{\ENCLOSE}[1]{\left\{ #1 \right\}}
\newcommand{\Enclose}[1]{\left[ #1 \right]}
\newcommand{\enclose}[1]{\left( #1 \right)}

\title[]{Periodic double tilings of the plane}
\author[]{Francesco Nobili} 
\address{Universit\'a di Pisa, Dipartimento di Matematica, Largo Bruno Pontecorvo 5,
56127 Pisa, Italy}
\email{\url{francesco.nobili@dm.unipi.it}}
\author[]{Matteo Novaga} 
\address{Universit\'a di Pisa, Dipartimento di Matematica, Largo Bruno Pontecorvo 5,
56127 Pisa, Italy}
\email{\url{matteo.novaga@unipi.it}}
\author[]{Emanuele Paolini} 
\address{Universit\'a di Pisa, Dipartimento di Matematica, Largo Bruno Pontecorvo 5,
56127 Pisa, Italy}
\email{\url{emanuele.paolini@unipi.it}}

\date{\today. 
}

\begin{document}
\begin{abstract}
We study tilings of the plane composed of two repeating tiles of different assigned areas relative to an arbitrary periodic lattice.
We classify isoperimetric configurations (i.e., configurations with minimal length 
of the interfaces) both in the case of a fixed lattice or for an arbitrary periodic lattice.
We find three different configurations depending on the ratio between the assigned areas 
of the two tiles and compute the isoperimetric profile.
The three different configurations are composed of tiles with a different number of 
circular edges, moreover, different configurations exhibit a different optimal lattice.
Finally, we raise some open problems related to our investigation.
\end{abstract}

\maketitle
\tableofcontents

\section{Introduction} 
The Kelvin problem (\cite{Kelvin1887}) consists in finding a partition of the Euclidean space $\R^d$, for $d\ge 2$, into cells of equal volume so that the area of the surfaces separating them is as small as possible. In the planar case, the natural candidate minimizer is the Honeycomb tiling, and this problem goes under the name of \emph{Honeycomb conjecture}. After preliminary results, a complete solution to this conjecture was achieved by T. C. Hales in \cite{Hales01} (see also \cite{Fejes1,Fejer2,Fejer3,Fejer4,MorganChristopherGreenleaf1988,Morgan1999}). In higher dimensions, little is currently known. For instance, in three dimensions a possible candidate minimizer was proposed by Lord Kelvin, the so-called Kelvin foam. However, an example with lower (average) surface area was later found in \cite{WeairePhelan1996}.

A natural variant of this problem is the study of periodic partitions of the Euclidean space minimizing the surface area of the interfaces but with possibly \emph{unequal} cells. Configurations of this kind have been for instance commented in \cite{Morgan1999} and numerically tested in \cite{FortesTeixeira01,FortesGranerTeixeira02}.
In this case, and differently from the equal-cells scenario, optimal tilings are expected to show a more complicated structure. In particular, guessing minimal configurations as well as rigorously proving that, indeed, these guesses are minimizers is a challenging yet intriguing theoretical question. For instance, an expected feature is the presence of phase transitions as the mutual size of the cells varies. Also, we observe that considering unequal cells is also natural from the viewpoint of the stability of the Honeycomb tiling, namely understanding the shape of planar partitions allowing the cells to have almost all the same size (see, e.g.\ \cite{CarrocciaMaggi16}).

In \cite{NobiliNovaga24}, the first and second named author have 
considered isoperimetric tilings composed of unequal cells, and obtained general
existence and regularity properties. 
The goal of this work is to classify all periodic tilings of the plane 
composed of two tiles with different assigned areas and with minimal perimeter 
both relative to a fixed periodic lattice (Theorem~\ref{th:main 1}) and 
for a variable periodic lattice (Theorem~\ref{thm:main 2}).
We find three different configurations, depending on the relative area of the two tiles
(see Figure~\ref{fig:tilings}). 
Finally, we compute the isoperimetric profile for the optimal lattice (see Figure~\ref{fig:profile}).
\subsection{Statement of the main results} Let us start by introducing the relevant notation. Consider a $d$-dimensional group of translations $G$ in $\R^d$, i.e., a \emph{lattice}. 
A fundamental domain relative to $G$ is a measurable set $D\subset\R^d$ such that 
\[
    \abs{\R^d \setminus \bigcup_{g\in G} (D+g)}=0\qquad
    \abs{(D+g)\cap D}=0,\quad  \forall g\in G\setminus\ENCLOSE{0}
\]
One can easily show that the measure of a fundamental domain only depends
on the group $G$: we will call it the \emph{volume} of the lattice $G$. Elements $g_1,\dots, g_M$ of $G$ are called \emph{generators}
of the lattice $G$ if they are independent vectors of $\R^d$ and every $g\in G$ can be written 
as $g=\sum_{k=1}^M c_k g_k$ with $c_k\in \Z$. Notice that 
the parallelepiped $D=\sum_{k=1}^d c_k g_k \colon c_k\in [0,1]$ is a fundamental domain of $D$ with measure $\abs{\det(g_1,\dots,g_d)}$. This formalism has been effectively employed for the study of isoperimetric problems starting from \cite{Choe89} and subsequently in the works \cite{MartelliNovagaPludaRiolo2017,CesaroniNovaga22,NovagaPaoliniStepanovTortorelli22,NovagaPaoliniStepanovTortorelli23,CesaroniNovaga23_1,CesaroniFragalaNovaga23,NobiliNovaga24} dealing with different types of local and non-local perimeter functionals satisfying suitable group invariance properties.

\begin{figure}
\includegraphics[height=3cm]{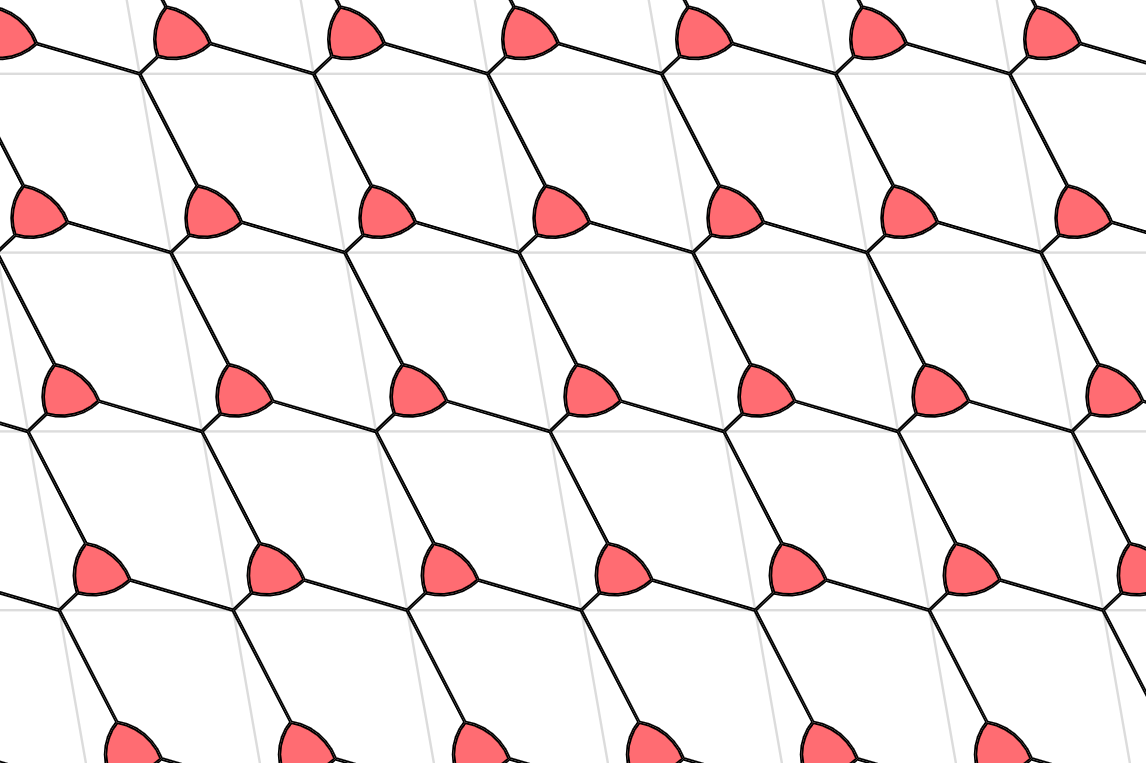}
\hfill
\includegraphics[height=3cm]{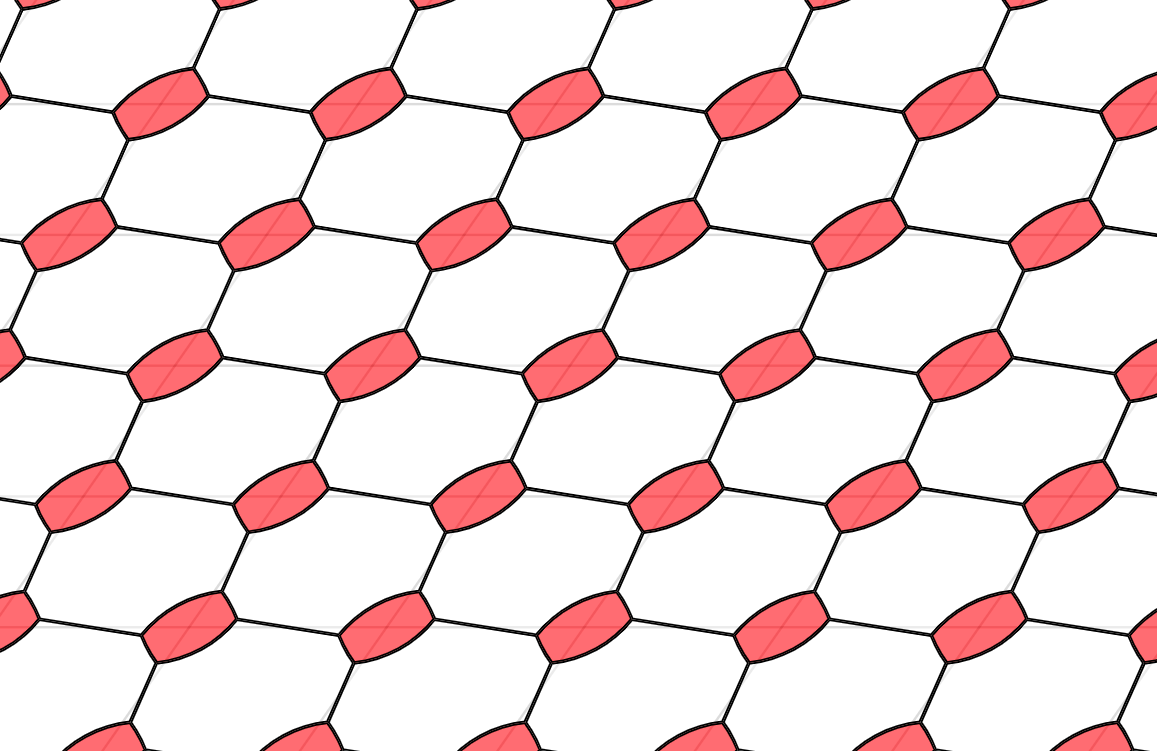}
\hfill
\includegraphics[height=3cm]{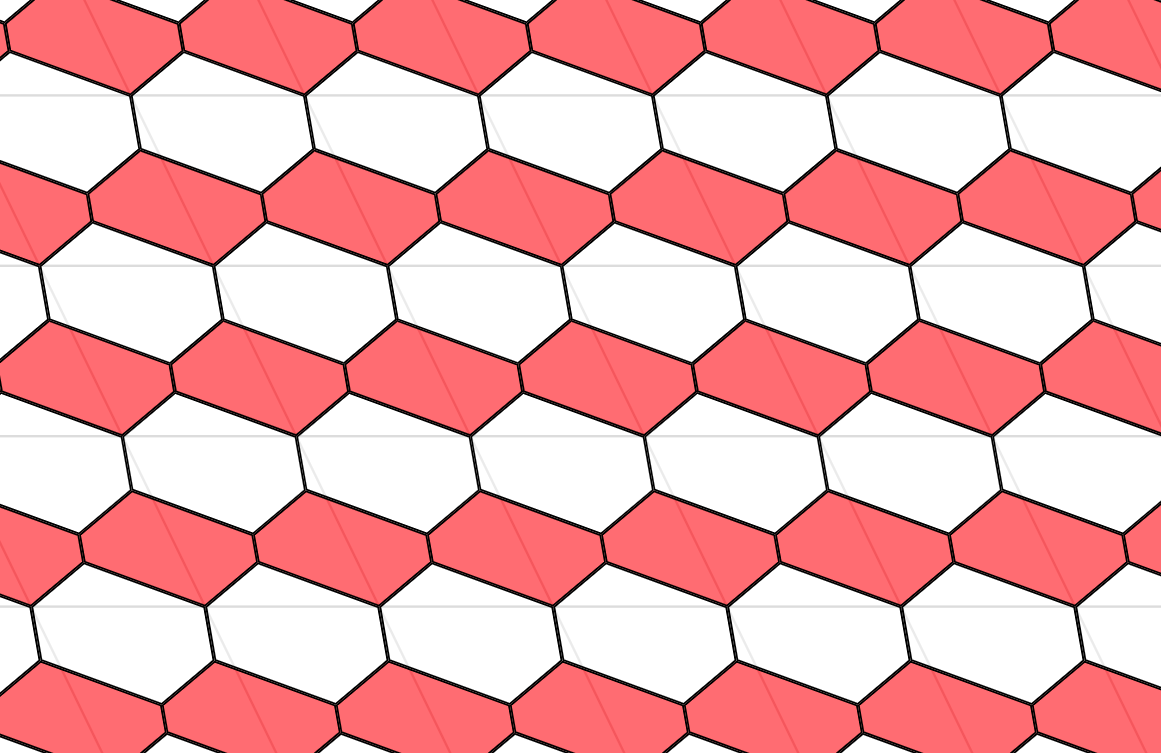}
\caption{The three possible configurations of isoperimetric $2$-tilings of the plane relative to a given periodic lattice.
See the interactive web page: \href{https://paolini.github.io/double-tiling/}{https://paolini.github.io/double-tiling/}.}
\label{fig:tilings}
\end{figure}

A \emph{tiling} in $\R^d$ is a collection $\T$ of measurable sets so that $\abs{E\cap F}=0$ for all $E,F\in \T$, $E\neq F$, and
$\abs {\R^d\setminus\bigcup_{E\in \T} E}=0$. Given $N\in\N$, a periodic $N$-tiling is a tiling $\mathbf T$ such that there is a lattice $G$ and there are $N$ measurable sets $E_1,...,E_N$ called \emph{generators} so that (up to negligible sets), we have
\[
\mathbf T = \{ E_k + g : k=1,...,N, g \in G\}.
\]
We call $\mathbf T$ non degenerate, provided $|E_k|>0$ for every $k=1,\dots,N$. We define the perimeter of a periodic tiling as
\[
\Per(\mathbf T) \coloneqq \frac 12\sum_{k=1}^N\Per(E_k),
\]
where $\Per(E)$ is the classical Caccioppoli perimeter of a measurable set 
(see \cite{AmbrosioFuscoPallarabook,Maggi12_Book}).
Notice that $\Per(\mathbf T)$ corresponds to the $d-1$ dimensional measure of all the mutual interfaces between the generators. Given a lattice $G$,
a periodic $N$-tiling 
$\mathbf T = \{E_k+g\colon k=1,\dots,N, g\in G\}$
is said to be \emph{isoperimetric} relative to $G$, provided that for any other periodic $N$-tiling 
$\mathbf T'=\{E'_k+g\colon k=1,\dots,N, g\in G\}$ 
with $\abs{E_k}=\abs{E_k'}$, $k=1,\dots N$ one has
\[
\Per(\mathbf T) \le \Per(\mathbf T').
\]
Finally, a periodic $N$-tiling $\mathbf T$ will be simply called isoperimetric, provided that  $\Per(\mathbf T)\le \Per(\mathbf T')$ for all periodic $N$-tiling $\mathbf T'$ with $\abs{E_k}=\abs{E_k'}$ without any restriction on the lattice. From now on, we will restrict our investigation to planar periodic $2$-tilings:
$d=2$, $N=2$.

The first main result is the following classification theorem for isoperimetric tilings of the plane relative to a lattice $G$ with unit volume (i.e., fundamental domains of $G$ have area $1$).
\begin{theorem}[classification of planar isoperimetric $2$-tilings]
    \label{th:main 1}
Let $\T$ be a non degenerate, periodic $2$-tiling of $\R^2$ which is isoperimetric relative to a fixed lattice $G$. If $E_1,E_2$ are two generators of the tiling, 
then $E_1$ and $E_2$ are curvilinear polygons with angles of 120 degrees,
there exists $p\ge 0$ (the pressure) such that, up to exchanging $E_1$ with $E_2$, 
we have one of the following configurations (see Figure~\ref{fig:configurations}):
\begin{itemize}
    \item[$(3;9)$] $p>0$, $E_1$ is a Reauleaux triangle with three arcs of curvature $p>0$, $E_2$ is a curvilinear ennagon which is a \emph{chipped hexagon} with six flat edges and three 
    concave edges with curvature $-p$;
    \item[$(4;8)$] $p>0$, $E_1$ is a strictly convex curvilinear quadrangle with the symmetries of a rectangle, the four arcs have curvature $p$,
    $E_2$ is a curvilinear octagon which is a \emph{chipped parallelogram} with 
    four flat sides and four concave edges with curvature $-p$;
    \item[$(6;6)$] $p=0$, $E_1$ and $E_2$ are two hexagons with 
    flat edges.
\end{itemize}
\end{theorem}
The proof of the above is a direct consequence of Theorem~\ref{th:equal_pressure}, dealing with preassure $p=0$ and Theorem ~\ref{th:different_pressures}, dealing with preassure $p>0$. In both results, we also comment on the uniqueness of isoperimetric configurations.

We next discuss our second main result. Given a lattice $G$ of $\R^2$ with unit volume we denote by $\mathcal I_G\colon [0,1] \to [0,+\infty)$ the isoperimetric \emph{profile function} relative to $G$ defined, for each $x \in [0,1]$, by:
\[
\mathcal I_G (x) \coloneqq \inf\left \{ \Per(\mathbf T) \colon \mathbf T = \{E_k + g \colon k=1,2, g \in G \}, |E_1|=x,|E_2|=1-x\right\}.
\]
Allowing the lattice $G$ to vary, we can introduce a further isoperimetric profile function setting
\[
\mathcal I (x) \coloneqq \inf \mathcal I_G(x) 
\]
where the infimum is taken among all lattices $G$ of $\R^2$ with unit volume. We are ready to state our second main results of this work (see also Figure \ref{fig:profile} for a picture).
\begin{theorem}\label{thm:main 2}
    There are explicit points $x_1,x_2 \in \left(0,\frac 12\right)$ (approx. $x_1 \approx 0.062,x_2\approx 0.317$, see Theorem \ref{thm:isoperimetric_profile}) so that:
    \[
        \mathcal I(x) = \begin{cases}
            \,\sqrt 2\sqrt{\pi -\sqrt 3}\sqrt{x}+\sqrt[4]{12},  &\text{if }x\in[0,x_1),\\
            \,2\sqrt{\frac{\pi}{3}+1-\sqrt 3}\sqrt x +2&\text{if }x\in[x_1,x_2),\\
            \,2\sqrt[4]{3} &\text{if }x\in(x_2,\frac 12],\\
            \mathcal I (1-x) &\text{if }x\in(\frac 12,1].
        \end{cases}
    \]
    In particular, we have that $[0,1]\ni x\mapsto \left(\mathcal I(x)-\mathcal I(0)\right)^2$ is concave.
\end{theorem}
The key content of the above result is that both the values of the points $x_1,x_2$ as well as the expression of $\mathcal I(\cdot)$ in the three successive sub-intervals are explicit. In fact, Theorem \ref{thm:main 2} is a direct consequence of Theorem \ref{thm:isoperimetric_profile} where a complete description of minimizers is achieved on each interval $(0,x_1),(x_1,x_2),\left(x_2,\frac 12\right)$, see Figure \ref{fig:profile}. The value $\mathcal I(\cdot)$ is then precisely the perimeter-cost associated to each of the three configurations. Notice that at the extreme point $x=0$, the limiting isoperimetric periodic $1$-tiling minimizing $\mathcal I(0)$ is the honeycomb tiling, as expected. Similarly, at $x=\frac 12$, the isoperimetric periodic $2$-tiling is the honeycomb partition made of hexagons with area $\frac 12$.
\begin{figure}
\begin{tikzpicture}[scale=1.6]
    \node at (0,0) {\includegraphics[height=8cm]{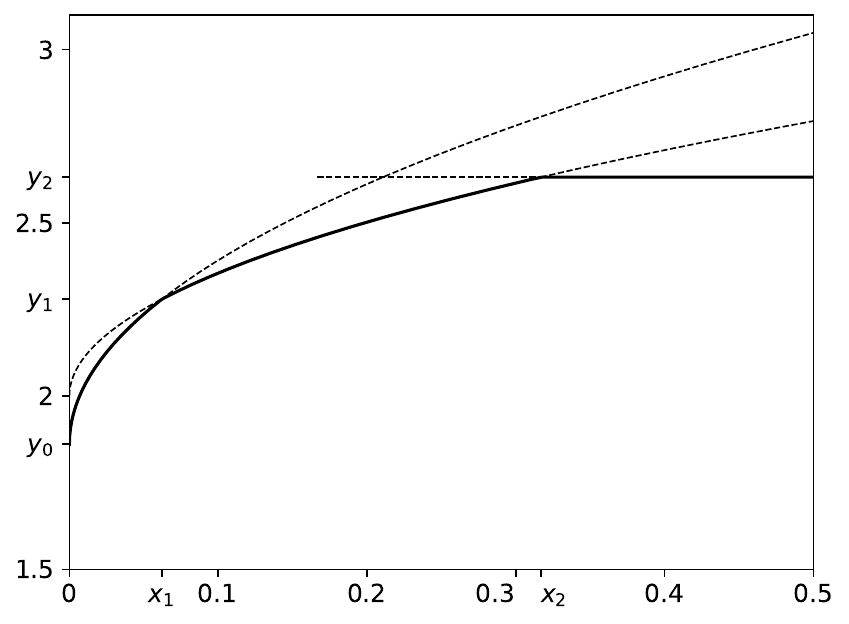}};
    \node at (-2.6,-1.4) {\includegraphics[height=1.6cm]{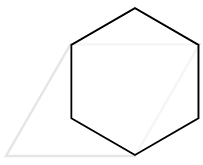}};
    \node at (-2,-0.4) {\includegraphics[height=1.6cm]{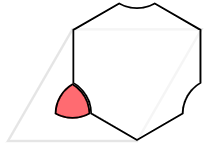}};
    \node at (-1,-0.4) {\includegraphics[height=1.6cm]{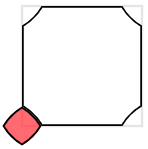}};
    \node at (0.7,0.5) {\includegraphics[height=1.6cm]{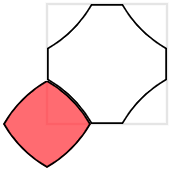}};
    \node at (0.7,-0.5) {\includegraphics[height=1.6cm]{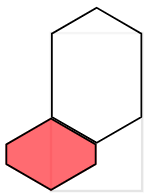}};
    \node at (2.5,0.5) {\includegraphics[height=1.6cm]{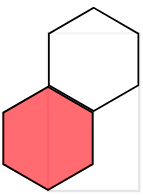}};
\end{tikzpicture}
    \caption{Plot of the isoperimetric profile $\mathcal I(x)$ where $x\in[0,\frac 12]$ is the area of the smaller tile and $1-x$ is the area of the larger tile.
    On the $y$-axis, we have $y_0 \coloneqq  \mathcal I(0) = \sqrt[4]{12}$, and $y_i \coloneqq \mathcal I(x_i)$ for $i=1,2$ (see Theorem~\ref{thm:main 2} and Theorem~\ref{thm:isoperimetric_profile}).
    The dashed lines represent stationary but not optimal tilings.
    The shape of the optimal tiles is depicted for the values $x=0$ (single hexagon), 
    $x=x_1$ (configurations $(3;9)$ and $(4;8)$, $x=x_2$ (configurations $(4;8)$ 
    and $(6;6)$)
    and $x=0.5$ (configuration $(6;6)$). 
    The transition between the three different configurations 
    occurs at $x=x_1$ and $x=x_2$. These are the only values where we have two different
    isoperimetric tilings with the same areas.}
    \label{fig:profile}
\end{figure}
Finally, we point out that it is also implied by the work \cite{NobiliNovaga24} of the first and second named authors that two adjacent hexagons with almost the same area are isoperimetric $2$-tilings, see \cite[Theorem 5.3]{NobiliNovaga24}. However, this was achieved by a \emph{qualitative} argument and the explicit value of the transition point $x_2$ could not be deduced (even though more than $2$-cells are allowed).
\subsection{About the proof scheme}
Our main results are deduced by the combination of theoretical existence and regularity results involving the theory of sets of finite perimeter, planar classification arguments, and ad-hoc trigonometric computations. 

The regularity properties of isoperimetric tiles are the same as those of planar bubbles
as stated in \cite{NobiliNovaga24} (see Theorem~\ref{th:regularity} below).
In a nutshell, the relevant regularity information is that the interfaces of an isoperimetric periodic tiling are either straight line segments or circular arcs meeting at triple points with equal angles of 120 degrees. 
Moreover, circular arcs have equal curvature up to a sign convention. 

The main difficulty is now given by the fact 
that it is not a priori known if the single tiles of an isoperimetric tiling are connected.
In particular we have to deal with an adjacency graph which, in principle, can be arbitrarily complex.
Starting from the regularity properties, 
the key points to deduce Theorem~\ref{th:main 1}, 
are stated in Proposition~\ref{prop:planar-tilings}, where periodicity is crucial, and Proposition~\ref{prop:planar-double-tilings} which uses the assumption that we have 
only two different tiles.

Finally, having classified three possible isoperimetric configurations, we perform ad-hoc trigonometric computations to compute the associated perimeter-costs. The value of $\mathcal I(x)$ for $x \in \left(0,\frac 12\right)$ is then obtained by means of geometric observations building on top of the previous computations. 
This step of the strategy gives also the explicit values of $x_1,x_2$ and, ultimately, completes the proof of Theorem \ref{thm:main 2}.

\subsection{Comparison with related works and open problems}
Understanding the shape of isoperimetric configurations is a long-standing research program that intersects several areas of mathematics (see, e.g., \cite{Fusco15Survey} and references therein for a modern introduction). 
The theory of clusters is closely related to our investigation with respect to the methodology and the techniques (see \cite{Maggi12_Book,MorganBOOK}). 
A relevant problem is the so-called $k$-bubbles problem consisting of separating $k$ given volumes with the least interface area, in a $d$-dimensional space. 
If $k=1$, this problem reduces to the well-known isoperimetric problem. 
For $k\le d+1$, optimal configurations are conjectured to be standard $k$-bubbles (\cite{SullivanMorgan96}) and the relevant advances around this program are contained in \cite{FoisyAlfaroBrockHodges93} for planar $2$-bubbles (see also \cite{CicaleseLeonardiMaggi17,MorganWichiramala02,DorffLawlorSampsonWilson09}), 
 in \cite{HassHutchingsSchlafly95,HassSchlafly00,HutchingsMorganRitore02} for $2$-bubbles in $\R^3$, and \cite{Reichardt08} in $\R^d$. For the resolution with a higher number of bubbles, we refer to \cite{Wichiramala} (planar $3$-bubbles), to \cite{PaoliniTamagnini18,PaoliniTamagnini18bis,PaoliniTortorelli20} (planar $4$-bubbles) and to the authoritative recent works \cite{MilmanNeeman22,MilmanNeeman24}. 
 See also \cite{DeRosaTione23} for recent advances in the stationary case,
and \cite{VazCoxAlonso04} for bidispersed bubble clusters. 

We conclude this introduction by mentioning possible investigations as well as interesting open problems related to this work, see Section \ref{sec:open problems} for a detailed discussion. For instance, a natural question is to explicitly find the profile $\mathcal I_G$ for every lattice $G$ in $\R^2$ and understand when the intermediate stationary configuration $(4;8)$ is actually minimal (see Problem \ref{prob:G exists} and Problem \ref{prob:IG explicit}). A further interesting problem is what we call the Kelvin problem with unequal cells. It is to be understood if our isoperimetric tilings are local minimizer for the perimeter in the plane (see Definition \ref{def:constrained local minimizer}). 
A simple computation we carried out in~\eqref{eq:kelvin} reveals that, in some range of areas, the periodic isoperimetric tiling cannot be local minimizer for the perimeter. 
However, for some other values of the areas, this could be the case (see Figure \ref{fig:comparison}).

\section{Classification of planar isoperimetric tilings}
In this section we classify planar periodic $2$-tilings which are isoperimetric relative to a fixed lattice $G$ 
(see Theorems~\ref{th:equal_pressure} and~\ref{th:different_pressures}).
We start by collecting some regularity results.
\begin{theorem}[regularity]
    \label{th:regularity}%
Let $\T$ be a periodic $N$-tiling of $\R^2$ isoperimetric relative to a lattice $G$,
and let $E_1,\dots,E_N$ be a set of generators.
Without loss of generality, we can suppose that the regions $E\in \T$ are closed sets
coinciding with the closure of their interior, 
that the measure theoretic boundary coincides with the topological boundary
and that the perimeter is the length of the boundary. 
Also the indecomposable components of a region
(a measurable set $E$ is decomposable if $E=E_1\cup E_2$ with $\abs{E_1\cap E_2}=0$ and $P(E)=P(E_1)+P(E_2)$)
coincide with the topological connected components of its interior.

Moreover, the following properties hold.
\begin{enumerate}
    \item Each region $E\in \T$ is bounded, and it is the union of a finite number 
of components $C_1,\dots,C_n$ with connected interior.
Each component $C_k$ is a curved polygon in the sense that its boundary is composed by 
a finite number of arcs of circle or straight line segments (we call them \emph{edges}). 
\item The edges meet in triple points with equal angles of 120 degrees, 
the three signed curvatures of the edges joining in the triple point have zero sum.
\item It is possible to define the pressure vector $\vec p=(p_1,\dots,p_N)$
which is a vector of numbers $p_k\in \R$ such that the arc separating two different
regions $E_k+g$ and $E_j+h$, with $g,h\in G$, has curvature $\abs{p_k-p_j}$
and the region on the convex side is $E_j+h$ when $p_j\ge p_k$;
pressures are defined up to a constant, 
so we can suppose that the sum of all the pressures is zero.
\item The interface $\partial \T \coloneqq \displaystyle\bigcup_{E\in \T}\partial E$ is connected and every component of $E\in \T$ is simply connected.
\end{enumerate}
\end{theorem}
\begin{proof}
    Apart from property (4) these results are presented, with a different notation, 
    in \cite[Theorem 5.2]{NobiliNovaga24}.
    Indeed, properties (1) up to (3) have been derived for partitions obtained as minimizers 
    of the following variational problem:
    \[
     \inf_G \inf_D \inf \left\{ \frac12\sum_{k=1}^N\Per(E_k) \colon \begin{array}{ll}
        E_k\subseteq D\text{ Borel}, &|E_k|=v_k,\quad k=1,...,N\\
         |E_k\cap E_l|=0, &\big|D\setminus \cup_k E_k \big|=0,
    \end{array}  \right\},
   \]
   where the first infimum is taken among all lattices $G$ and the second infimum is taken 
   among all \emph{fundamental domains} $D$ for the action of $G$ and the third infimum is taken 
   among all Borel partitions $(E_i)$ of $D$ with assigned volumes $v_1,..,v_N$. 
   Evidently, $G,D,(E_k)$ minimizes the above if and only if 
   $\T \coloneqq \{ E_k + g \colon k=1,\dots,N, g \in G\}$ is an isoperimetric periodic 
   $N$-tiling of $\R^2$.

   Even though this was not stated, the regularity results proved in \cite[Theorem 5.2]{NobiliNovaga24} 
   actually extend to the variational problem where the lattice $G$ is fixed:
    \begin{equation}
     \inf_D \inf \left\{ \frac12\sum_{k=1}^N\Per(E_k) \colon \begin{array}{ll}
        E_k\subseteq D\text{ Borel}, &|E_k|=v_k,\quad k=1,...,N\\
         |E_k\cap E_l|=0, &\big|D\setminus \cup_k E_k \big|=0,
    \end{array}  \right\}.
   \label{eq:functional G}
   \end{equation}
   Indeed, \cite[Theorem 5.2]{NobiliNovaga24} relied on \cite[Lemma 4.2]{NobiliNovaga24} 
   (which clearly can be extended to minimizers of \eqref{eq:functional G}, 
   see \cite[Theorem 3.2]{NobiliNovaga24}) 
   and also on \cite[Proposition 5.1]{NobiliNovaga24}. 
   The latter can be readily checked to hold for minimizers of \eqref{eq:functional G} by observing, 
   as done for \cite[Proposition 5.1]{NobiliNovaga24}, 
   that a minimizer $D, (E_k)$ exists satisfying
    \[
     D \cap B_{r_G}(0) \neq \varnothing,\qquad \text{and}\qquad  {\rm diam}(D)\le 2r_G +C,
    \]
    where $C>0$ depends on $G,N$ (in fact, it depends on the perimeter of the parallelogram spanned 
    by a basis of $G$). 
    In particular, as before but being now $G$ fixed, we observe that $D,(E_k)$ is a minimizer 
    of \eqref{eq:functional G} if and only if 
    $\T\coloneqq \{ E_k + g\colon k=1,\dots,N, g\in G\}$ is a periodic 
    $N$-tiling isoperimetric relative to $G$.
    Hence, by the previous discussion, properties (1), (2), and~(3) hold true also for any 
    periodic $N$-tiling isoperimetric relative to $G$.

   Finally, to prove (4), just notice that if, by contradiction, 
   the interface $\partial \T$ were not connected, 
   one can take a connected component of $\partial \T$ and consider all the components contained in it and all their 
   translations by $G$. 
   We can then move these components with a translation obtaining a new tiling with the same 
   volumes and perimeter. At some point, the moving components will touch the rest of the interface 
   creating a new vertex which will violate property (2) and thus leading to a contradiction.
\end{proof}

From now on we will use the properties stated in Theorem~\ref{th:regularity},
in particular, we will always assume that the regions are closed sets,
without further reference. 

\begin{proposition}[properties of planar tilings]
\label{prop:planar-tilings}
Let $\T$ be a periodic $N$-tiling of $\R^2$ which is isoperimetric relative to a  lattice $G$.

If $E\in \T$ and $g\in G$ are such that $E$ and $E+g$ share a common edge, 
and if $C$ is the connected component of $E$ which has the edge 
in its boundary, then the component of $E+g$ which shares the edge is $C+g$.
\end{proposition}
\begin{proof}
Just notice that if the common edge is shared by the translations of two different connected components, we could remove the edge and decrease the perimeter. 
Formally let $E=E_k+h$, $C=C_k+h$ where $E_k$ is the common generator of $E$ and $E+g=E_k+h+g$ and 
$C_k$ is the corresponding connected component of $E_k$.
Suppose that $D_k$ is the connected component of $E_k$ such that $C=C_k+h$ and $D=D_k+h+g$ are separated by the common edge $e$. 
Then we could replace the generator $E_k$ with
\[
 E_k' = (E_k \setminus D_k)\cup(D_k+g)
\]
to obtain a new tiling $\T'$ such that $\Per(\T') = \Per(\T) - \HH^1(e)$.
\end{proof}
\begin{proposition}[properties of planar $2$-tilings]
\label{prop:planar-double-tilings}
Let $\T$ be a non degenerate periodic $2$-tiling of $\R^2$
isoperimetric relative to a lattice $G$.
Assume that the two pressures have zero sum: $p_1=-p_2 \ge 0$.
Let $E_1,E_2\in \T$ be two generators. 
Then the following hold.
\begin{enumerate}
\item[(i)] If $C$ is a connected component of a tile $E \in \T$
and $C$ has a flat edge $e$ then $C$ has another flat edge parallel to $e$ with 
opposite normal vector (we call it an \emph{opposite edge}).
\item[(ii)] If $C$ is a convex connected component of a tile $E \in \T$, 
the edges of $C$ are either all flat (zero curvature), 
or else all the edges are curved (have positive curvature).
\end{enumerate}
\end{proposition}
\begin{proof}
Let $C$ be a connected component of a tile and suppose that it has a flat edge $e$.
This means that the other component sharing the edge $e$ has the same pressure: it can either be a translation of the same component, or else it is a component of the translation of a different generating tile which has the same pressure.
It cannot be a translation of a different component of the same tile because of Proposition~\ref{prop:planar-tilings}.
If $e$ is shared by the translation $g\in G$ of the same component $C$, it means that $e-g$ is another edge of the same component which, of course, is parallel to $e$.
Otherwise, the two regions have the same pressure $p=-p$ and hence $p=0$. 
In the case $p=0$ all the components of all the tiles are straight line segments with angles of 120 degrees, so, all the components are hexagonal polygons with edges parallel to the edges of a common regular hexagon: 
hence they are composed of three pairs of opposite parallel edges.

For the second point suppose now that $C$ is also convex and, by contradiction,
suppose that it has a flat edge $e$ and a curved edge $f$.
For the previous point, the edge $e$ has a parallel edge $e'=e+g$ with $g\in G$. 
We can then follow the normal vector starting from $e$, passing through $f$ and arriving 
at $e'$. Along this path, the normal vector makes a total turn of 180 degrees and it never changes the direction of rotation (because $C$ is convex). On every vertex, the normal vector makes a jump of 60 degrees. This means that there cannot be more than two vertices because otherwise 180 degrees are spanned within three jumps of 60 degrees and all the edges in between $e$ and $e'$ must be flat, contrary to our assumption. Hence the curved edge $f$ is the only edge between $e$ and $e'$ and it is a circular arc spanning an angle of 60 degrees.

If $v$ is the vertex shared by $e$ and $f$, the other vertex of $f$ must be $v+g$ because it is shared with $e'=e+g$.
Let $D$ be the component on the other side of the curved edge 
$f$ of $C$. 
Following the edges of $D$ we notice that the third edge in the vertex $v+g$ (apart from $e'$ and $f$) 
must be $f+g$ and then we arrive at $v+2g$ and repeat the process. 
It turns out that the boundary of $D$ never closes and hence $D$ is unbounded. 
This is a contradiction since $D$ has finite area. 
\end{proof}

\begin{figure}
    \begin{center}
    \includegraphics[height=2.5cm]{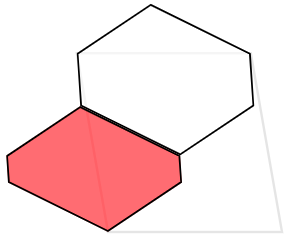}
    \qquad
    \includegraphics[height=2.5cm]{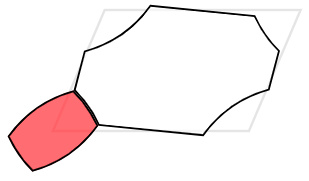}
    \includegraphics[height=2.5cm]{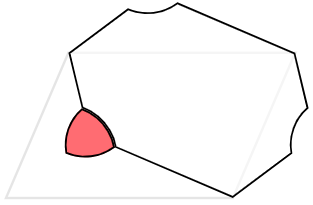}
    \end{center}
    \caption{The three possible configurations for an isoperimetric 
        $2$-tiling relative to a fixed lattice.
    The first one is configuration (6;6) where the two tiles
    are two hexagons with straight edges (equal pressures);
    the second one is configuration (4;8) where the tile with larger pressure is 
    a curved rectangle and the tile with lower pressure is a chipped parallelogram 
    i.e., an octagon with alternating flat and concave edges;
    the third one is configuration (3;9) where the tile with larger pressure is
    a Realeaux triangle and the tile with lower pressure is a chipped hexagon which is 
    an ennagon (9 edges) with 6 flat edges and 3 concave edges.
    } 
    \label{fig:configurations}
\end{figure}

\begin{theorem}[classification of $2$-tilings with equal pressures]
\label{th:equal_pressure}%
Let $\T$ be a non degenerate, periodic $2$-tiling of $\R^2$
isoperimetric relative to a lattice $G$.
Let $E_1$ and $E_2$ be two generators of the tiles.
Suppose that the two regions have equal pressures $p_1=p_2=0$. 
Then both $E_1$ and $E_2$ are hexagons with straight edges (in particular they are connected).
We will label this configuration as $(6;6)$ standing for the number 
of edges of the two tiles. See Figure~\ref{fig:configurations}.
Given $G$, $\abs{E_1}$ and $\abs{E_2}$, the configuration is uniquely 
determined up to translations.
\end{theorem}
\begin{proof}
Since the regions have equal pressures
we know that all the edges must have zero curvature, 
and since all angles are equal to 120 degrees, all connected components of $E_1$ and $E_2$
are hexagonal. Consider one connected component $C$ of $E_1$. 

\emph{Step 1.} We claim that $C$ cannot be surrounded only by translations of $E_2$.
Otherwise, 
each of the six components of the translations of $E_2$ surrounding $C$ 
would share an edge with the previous and next one. 
By Proposition~\ref{prop:planar-tilings} 
this means that they are all different translations of 
the same component $D$ of $E_2$.
So they can be written as $D+g_k$ with distinct $g_1,\dots, g_6\in G$
choosen so that $D+g_1,\dots,D+g_6$ are ordered in counter-clockwise order 
around $C$.
Since two adjacent translations of $D$ share an edge, each edge of the hexagon $D$ 
must be a translation of the opposite edge, 
meaning that opposite edges of $D$ have equal lengths.
The vectors representing the ordered edges of $D$ can hence be denoted by
$v_1,v_2,v_3,-v_1,-v_2,-v_3$.
Since $C$ has the edges in common with the translations of $D$
the opposite edges occur in $C$ with the same order: $-v_1,-v_2,-v_3,v_1,v_2,v_3$.
This means that also $C$ is a translation of $D$:
$C=D+g_1+g_3-g_2$. 
This means that $C$ is contained in a translation of $E_2$ which is a contradiction 
since $C$ was chosen to be a connected component of $E_1$.

\emph{Step 2.} We claim that the opposite edges of $C$ have equal length.
By the previous step and Proposition~\ref{prop:planar-tilings} we know that $C$ is adjacent to a translation $C+g$ of itself, with $g\in G$. 
So at least one couple of opposite edges have equal length.
Let $v_1,\dots,v_6$ be the ordered vectors identifying the edges of 
$C$ so that $v_1+\dots+v_6=0$. 
We know that $v_k$ is parallel to $v_{k+3}$ (with opposite direction)
for $k \in \Z$ (mod 6).
Suppose that $v_6=-v_3$ are the vectors corresponding to the two opposite sides with equal length
while $v_4=-\alpha v_1$ and $v_5=-\beta v_2$.
From $v_1+v_2-\alpha v_1-\beta v_2=0$ since $v_1$ and $v_2$ are independent 
we obtain $1-\alpha = 1-\beta = 0$ which means that also $v_4=-v_1$ and $v_5=-v_2$.
The claim follows.

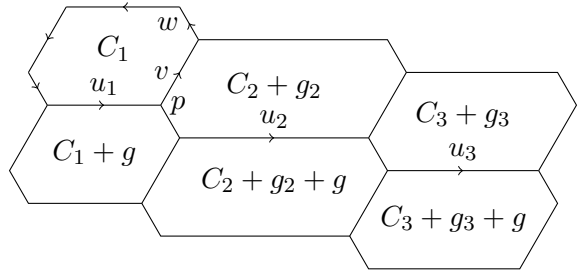
\begin{figure}
\begin{tikzpicture}[x=0.5cm,y=0.5cm]
\begin{scope}[
  decoration={
    markings,
    mark=at position 0.5 with {\arrow{>}}}]
  \draw[postaction={decorate}](0,0) -- ++(0:3) node[midway,above] {$u_1$};
  \draw[postaction={decorate}](3,0) -- ++(60:2) node[midway,left] {$v$};
  \draw[postaction={decorate}](3,0) ++(60:2) -- ++(120:1) node[midway,left] {$w$};
  \draw[postaction={decorate}](3,0) ++(60:2) ++(120:1) -- ++(180:3);
  \draw[postaction={decorate}](3,0) ++(60:2) ++(120:1) ++(180:3) -- ++(240:2);
  \draw[postaction={decorate}](3,0) ++(60:2) ++(120:1) ++(180:3) ++(240:2) -- ++(300:1);
  \draw(1.75,1.5) node {$C_1$};
  \draw(3,0) node[right] {$p$};

  \draw(0,0) -- ++(240:2) -- ++(300:1) -- ++(0:3) -- ++(60:2) -- ++(120:1);
  \draw(1.25,-1.25) node {$C_1+g$};

  \draw[postaction={decorate}](3,0) ++(-60:1) -- ++(0:5) node[midway,above] {$u_2$};
  \draw(8,0) ++(-60:1) -- ++(60:2) -- ++(120:1) -- ++(180:5);
  \draw(6,0.5) node {$C_2+g_2$};

  \draw(8,0) ++(-60:1) -- ++(-60:1) -- ++(-120:2) -- ++(180:5) -- ++(-240:1);
  \draw(6,-2) node {$C_2+g_2+g$};

  \draw[postaction={decorate}](8,0) ++(-60:2) -- ++(0:4) node[midway,above] {$u_3$};
  \draw(12,0) ++(-60:2) -- ++(60:2) -- ++(120:1) -- ++(180:4);
  \draw(11,-0.25) node {$C_3+g_3$};

  \draw(12,0) ++(-60:2) -- ++(-60:1) -- ++(-120:2) -- ++(180:4) -- ++(120:1);
  \draw(10.75,-3) node {$C_3+g_3+g$};
  
\end{scope}
\end{tikzpicture}
\caption{The construction of the hexagons in the proof of Theorem~\ref{th:equal_pressure}, Step 3.}
\label{fig:hexagons}
\end{figure}

\emph{Step 3.} Conclusion (see Figure~\ref{fig:hexagons}).
Consider a connected component $C_1$ of $E_1$ and let $g\in G$ be the translation
given in Step 1 such that $C_1+g$ is adjacent to $C_1$ itself. 
Let $p$ be one of the two vertices in common between $C_1$ and $C_1+g$. 
We claim that there is a connected component $C_2$ of $E_2$ and $g_2\in G$ such that $C_2+g_2$ has a vertex in $p$. If, by contradiction, it holds that $C_2\subset E_1$, then also $C_2=C_1$ (by Proposition \ref{prop:planar-double-tilings}) and the whole plane would be covered by $C_1+g$ with $g\in G$. So $\abs{E_1\setminus C_1}=0$ and $\abs{E_2}=0$, 
against the assumption that the tiling is not degenerate. So $C_2$ must be a connected component of $E_2$.
Two edges of the component $C_2+g_2$ are in common with $C_1$ and $C_1+g$ and hence are represented 
by two vectors $-v,-w$ where $v,w$ represent consecutive edges of $C_1$ with $v+w=-g$ (recall that all connected components are hexagons with equal angles of 120 degrees and opposite edges with equal length).
This means that $C_2+g_2+g$ is adjacent to $C_2+g_2$.
By repeating the reasoning with $C_1, C_1+g$ replaced by $C_2+g_2,C_2+g_2+g$ 
we find $C_3$, a connected component of $E_1$, and $g_3\in G$ such that 
$C_3+g_3$ has a common vertex with $C_2+g_2$ and $C_2+g_2+g$.
Moreover, $C_3+g_3+g$ is adjacent to $C_3+g_3$.
Now let $u_1,v,w,-u_1,-v,-w$ be the vectors representing the ordered edges of $C_1$,
let $u_2,v,w,-u_2,-v,-w$ be the ordered edges of $C_2$, 
and $u_3,v,w,-u_3,-v,-w$ be the ordered edges of $C_3$.
Notice that $u_1=t_1 u$, $u_2=t_2 u$, $u_3 =t_3 u$ all have the same direction $u$.
The hexagons $C_1+kg$, $C_2+g_2+kg$, $C_3+g_3+kg$ with $k\in \Z$ 
form three parallel strips with direction $g$ whose interfaces are zig-zags following 
the vectors $v$ and $w$. 
We can swap the second and third strip to obtain a new tiling with the same 
areas and with a whole interface removed. 
This is contrary to the minimality of $\T$ unless 
$C_3=C_1=E_1$.
Analogously one finds $C_2=E_2$.
Hence $E_1$ and $E_2$ are both connected.
Uniqueness is straightforward, concluding the proof.
\end{proof}
\begin{figure}
\begin{tikzpicture}[line cap=round,line join=round,>=triangle 45,x=1.0cm,y=1.0cm]
\clip(-0.57,-1) rectangle (4.83,1.5);
\draw [line width=1.2pt] (0,0)-- (1,0);
\draw [shift={(1.75,-0.43)},line width=1.2pt]  plot[domain=1.57:2.62,variable=\t]({1*0.87*cos(\t r)+0*0.87*sin(\t r)},{0*0.87*cos(\t r)+1*0.87*sin(\t r)});
\draw [shift={(1.75,0.43)},line width=1.2pt]  plot[domain=3.67:4.71,variable=\t]({1*0.87*cos(\t r)+0*0.87*sin(\t r)},{0*0.87*cos(\t r)+1*0.87*sin(\t r)});
\draw [line width=1.2pt] (3,0)-- (4,0);
\draw [line width=1.2pt] (4.5,0.87)-- (4,0);
\draw [line width=1.2pt] (4,0)-- (4.5,-0.87);
\draw (0.72,0.4) node {$2$};
\draw (0.75,-0.4) node {$2$};
\draw (1.5,0.0) node {$1$};
\draw (3.73,0.4) node {$2$};
\draw (3.78,-0.4) node {$2$};
\draw (4.4,0.0) node {$2$};
\begin{scriptsize}
\fill (1,0) circle (1.5pt);
\fill (4,0) circle (1.5pt);
\end{scriptsize}
\end{tikzpicture}
\caption{The two possible vertices for the tilings in the assumptions of Proposition~\ref{prop:positive_tiling}.
    The first case is a vertex of type 1-2-2, and the second is a vertex of type 2-2-2.}
\label{fig:vertices}
\end{figure}
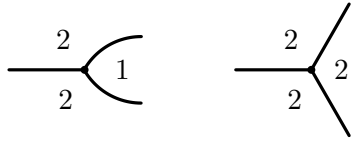
\begin{proposition}[properties of $2$-tilings with different pressures]
\label{prop:positive_tiling}%
Suppose that $\T$ is a periodic, non-degenerate, $2$-tiling of $\R^2$,
isoperimetric relative to a lattice $G$,
generated by $E_1,E_2$ with pressures $p_1=-p_2>0$.
\begin{enumerate}
\item[(i)] The connected components of $E_1$ are strictly convex.
\item[(ii)] Consider a vertex $v$ in $\partial \T$ and the three 
edges joining in $v$. We have two possibilities
(see Figure~\ref{fig:vertices}):
either the three edges
are flat and the three components around $v$ are different translations of the same connected
component of the region $E_2$
(vertex of type 2-2-2)
or else there is one flat edge separating two different translations of the same connected component of the region $E_2$
and two curved edges, symmetric with respect to the line containing the flat edge, 
which are consecutive edges of a translation of a component of $E_1$ (type 1-2-2).
\end{enumerate}
\end{proposition}
\begin{proof}
Let $C$ be a component of $E_1$. Since $p_1>0$ we know that 
$C$ is convex. 
We now claim that $C$ is strictly convex.
By Proposition~\ref{prop:planar-double-tilings} (ii) the edges of $C$ are either all strictly curved or else all flat. Suppose by contradiction they are all flat: then $C$ is a hexagon and crossing each of the edges we find a translation of $C$ (we cannot find the other region, because they have different pressures) and hence $C$ is actually covering the whole plane. This means that $\abs{E_2}=0$, against our assumptions. So $C$ is strictly convex.

Each vertex of the tiling has order three, so we have three components around it. If one of the components belongs to a translation of $E_1$, it must be strictly convex, by the previous point. In this case, the other two components have a strictly concave edge and hence are components of translations of $E_2$.
We know that the sum of the three curvatures of the three edges must be zero. The two curved edges have the same curvature and their contribution cancels in the sum. Hence the third edge must be flat and we have a vertex of type 1-2-2.

Otherwise, all the three components are translations of components
of $E_2$, separated by flat edges. By Proposition~\ref{prop:planar-tilings} we know that they must be different translations of the same component of $E_2$. Hence the vertex is of type 2-2-2.
\end{proof}

\begin{theorem}[classification of $2$-tilings with different pressures]\label{th:different_pressures}
Let $\T$ be a non-degenerate $2$-tiling of $\R^2$, 
isoperimetric relative to a fixed lattice $G$, 
let $E_1,E_2\in \T$ be two generators of the tiling 
with corresponding pressures $p_1 = -p_2 >0$.
Then we have the following possibilities (see Figure~\ref{fig:configurations}):
\begin{enumerate}
\item[(i)]
 $E_1$ is a curvilinear rectangle (i.e., a curvilinear quadrangular 
 shape with the symmetries of a rectangle) and $E_2$ 
 is an octagonal curvilinear polygon with alternating flat and concave edges (we will label this configuration as $(4;8)$ standing 
 for the number of edges of $E_1$ and $E_2$);
 \item[(ii)]
 $E_1$ is a Reuleaux triangle (i.e., a curvilinear triangle with equal angles of 120 degrees and equal curved edges) 
 and $E_2$ is a \emph{chipped hexagon} which is 
 a ennagonal ($9$-edges) curvilinear polygon 
 with 6 flat edges and 3 curved concave edges 
 (we will label this configuration as $(3;9)$ standing for the number of edges of $E_1$ and $E_2$).
\end{enumerate}
Given $G$, $\abs{E_1}$ and $\abs{E_2}$, the configuration $(4;8)$
is uniquely determined up to translations. 
Finally, if the hexagon $E_2$ is regular (invariant by appropriate rotations of 60 degrees), also the configuration (3;9) is unique up to isometries.
\end{theorem}
\begin{proof}
Let $C$ be a connected component of $E_2$.
The edges of $C$, separating the component $C$ from a translation of $E_1$,
must have negative curvature $p_2$
while the edges separating $C$ from a translation of $E_2$ must be flat.
Each flat edge has a corresponding parallel edge, hence the number of flat edges in $C$ 
is even.

We now claim that there are at most $3$ pairs of parallel flat edges in $C$.
Let us consider an abstract planar graph $H$ whose vertices are the elements of the group $G\subset \R^2$
and whose edges are given by the pairs $(g_1,g_2)$ such that $C+g_1$ is adjacent 
to $C+g_2$.
The graph $H$ is invariant under the action of $G$ on itself.
Notice that $G$ is an abelian group with two generators, hence $G$ is isomorphic to 
the group $\Z^2$.
Lemma~\ref{lm:eulero} proves that the vertices of $H$ must have order (number of edges per vertex)
not larger than 6. 
This means that the connected component $C$ has at most six flat edges.

By Proposition~\ref{prop:positive_tiling} at every vertex of the tiling we 
have either three flat edges (vertex of type 2-2-2) 
or one flat edge and two curved ones (vertex of type 1-2-2) 
meeting in the vertex.
Hence the region between the two curved edges of a vertex is a translation 
of a connected component of $E_1$.
This means that each curved edge of $C$ is adjacent to two flat edges.

Since $C$ has negative pressure, each edge of $C$ has a non positive 
turning angle, and hence, applying Lemma~\ref{lm:turning}, we find that 
the left-hand side of \eqref{eq:turning} is non positive. 
Hence the number of edges of $C$ is $n\ge 6$.

Since we cannot have two consecutive curved edges it means that we have at 
least three flat edges. 
Since flat edges are an even number we have at least four of them.
But we already stated that we have at most six flat edges, hence 
we have either four or six.

Suppose now that $C$ has two pairs of opposite parallel flat edges.
We first claim that all the vertices of $C$ are of type 1-2-2.
Suppose by contradiction that we have a vertex of type 2-2-2. 
The two flat edges of $C$ in such a vertex have a corresponding parallel opposite edge
whose corresponding vertex must also be of type 2-2-2. 
So each of the two opposite flat edges has an adjacent flat edge which,
by taking into account the 120 degrees of the angles, is different from the others.
Thus we would have at least 6 flat edges, which we have excluded.
Since all vertices of $C$ are of type 1-2-2, and $C$ is a component of the region $E_2$,
it means that the edges of $C$ are alternating between flat and curved edges and
$C$ is a curvilinear octagon with four flat edges and four curved concave edges.
If we prolong the flat edges of $C$ we obtain a parallelogram $P$ with straight edges, 
which must be a fundamental domain of the group $G$ because on the other side of each edge of $P$
we have a translated copy of $P$ itself, and hence $\ENCLOSE{P+g\colon g\in G}$ is a tiling 
of $\R^2$.
This means that $E_2=C$ is a ``chipped parallelogram`` (with four flat edges and four concave edges) 
and $E_1$ is a curvilinear quadrangular polygon.
By Lemma~\ref{lm:quadrangular} we obtain that $E_1$ is a curvilinear rectangle

We have found the $(4;8)$ configuration.

Suppose now that $C$ has three pairs of opposite parallel flat edges.
If there are no other edges, $C$ is a hexagon which alone will 
tile the whole space. This means that $\abs{E_1\setminus C}=0$ and $\abs{E_2}=0$ 
which is excluded by the assumption that the tiling is non-degenerate.

Hence $C$ must have at least a vertex $v$ of type 1-2-2. 
The flat edge in this vertex has a corresponding opposite flat edge,
obtained by means of a translation $g\in G$,
which also must have a 1-2-2 corresponding vertex $w=v+g$.
Following the curved edge of $C$ in $v$ we arrive at another 
vertex of type 1-2-2, then following the straight edge of $C$ we arrive 
at a vertex which can be either of type 1-2-2 or of type 2-2-2.
In the first case we can follow the curved edge of $C$ to reach another 
vertex of type 1-2-2. Then, following the straight edge of $C$ we arrive 
in a vertex which is separated from $w$ by the curved edge of $w$.

By these considerations, we arrive at two possible configurations:
either all the vertices are of type 1-2-2 and we have a 
\emph{chipped hexagon} which is a curvilinear dodecagon with 6 flat edges 
alternated with 6 curved concave edges 
(we will label this configuration $(3,3;12)$)
or we have a hexagon with only three chipped vertices which 
is an ennagon with 6 flat edges each of which has 
two vertices of different type 1-2-2 and 2-2-2 so that every two 
flat edges we have a curved concave edge
(we will label this configuration $(3;9)$).

In both configurations the flat edges of $C$ can be prolonged to obtain 
a flat hexagon $P$ containing $C$ such that $\ENCLOSE{P+g\colon g\in G}$ 
is a tiling of $\R^2$. This means that $E_2=C$.
Each vertex of $P$ is shared by $P$ and by $P+g$ and $P+h$ with $g,h\in G$.
If the vertex of $P$ is not a vertex of $C$ we easily deduce that 
$P\setminus C$ is contained in a component of $E_1$ which 
is a Reauleaux triangle (i.e., a curvilinear triangle with three equal angles of 120 degrees).

In the configuration labeled $(3,3;12)$ we have in fact two such triangles,
and $E_1$ turns out to be disconnected while $C=E_2$ 
has 12 edges.
In the configuration $(3;9)$ we have that $E_1$ is a single 
Reauleax triangle and $E_2$ has 9 edges.

The configuration $(3,3;12)$ is a stationary configuration but we claim that it is not isoperimetric. 
In fact, it is possible to increase the size of the larger of the two Reauleaux triangle (or any of them if they are equal) and decrease 
the size of the other Reauleaux triangle 
so that the total area enclosed by the two triangles remains constant while 
the perimeter decreases.
This is because the contribution to the perimeter of a Reauleaux 
triangle (the difference between the perimeter of the triangle and the length of 
the corresponding tripod) is increasing and proportional to the linear size of the triangle,
while the enclosed area is proportional to the square of the linear size.
So, if we have two Reauleaux triangles of respective sizes 
$r$ and $s$, the total volume is kept fixed if $k=r^2+s^2$ is constant 
while the contribution to the total perimeter is proportional to
\[
  r + s = r + \sqrt{k-r^2}.
\] 
This function is concave and has a maximum for $r=s$ hence 
we can always decrease the perimeter by slightly enlarging the larger 
Reauleaux triangle and decreasing the smaller one.
\end{proof}
Notice that the configurations $(3;9)$ is, for a general lattice, not unique
because we can choose to place the Reauleaux triangle in two different vertices of the hexagon. Notice that Theorem \ref{th:equal_pressure} combined with Theorem \ref{th:different_pressures} directly implies our first main result Theorem \ref{th:main 1}.

\section{Optimal lattice}\label{sec:optimal_Tilings} 
In this section we identify the optimal lattice for isoperimetric periodic $2$-tilings, for 
all values of the two assigned areas.
We start with the degenerate case, namely when one of the two areas is zero.
The result is implied by Hales Theorem \cite{Hales01}, 
we nevertheless propose a proof which is much simpler since periodicity is enforced
(see also \cite{CesaroniNovaga22} for an extension to anisotropic perimeters).
\begin{theorem}[optimal periodic honeycomb]\label{th:honeycomb}
Let $\T$ be an isoperimetric periodic $1$-tiling of the plane.
Then $\T$ is generated by a single tile $E$ which is a regular 
hexagon, and
\begin{equation}\label{eq:hexagon}
  \Per(E) = \sqrt[4]{12}\sqrt{\abs{E}}.
\end{equation}
\end{theorem}
\begin{proof}
Since we have a single tile, 
the pressure must be $p=0$ and hence all connected components of the tile 
are hexagons with internal angles of 120 degrees.
Proposition~\ref{prop:planar-tilings} implies that $E$ is composed by a single 
connected component, hence $E$ itself is an hexagon.

Let $a,b,c$ be the length of three consecutive sides of the hexagon 
(opposite sides have the same length).
Then a straightforward computation gives
\begin{align*}
    \Per(\T) &= a+b+c\\
    \abs{E} &= \frac{\sqrt 3} 2 (ab + bc + ca).
\end{align*}
Using Lagrange multipliers it is easy to check that if $\Per(\T)$ 
is minimal with $\abs{E}$ fixed, then $a=b=c$ which 
means that $E$ is a regular hexagon.
\end{proof}
In the following three theorems, we handle the three other non degenerate configurations.
\begin{theorem}[optimal lattice for (6;6) configuration]\label{th:computation 66}
Let $\T$ be a non-degenerate, isoperimetric periodic $2$-tiling of the plane. If $\T$ is a $(6;6)$ configuration,
then the lattice admits a fundamental domain $D$ which is 
the union of two regular hexagons $E_1,E_2$ sharing an edge.
Moreover, one has 
\begin{align}
    \frac 16 &< \frac{\abs{E_i}}{\abs{E_1}+\abs{E_2}}<\frac 56,\qquad \text{for }i=1,2.\label{eq:vol 66}\\
   \Per(\T) &= 2\sqrt[4]{3}\sqrt{\abs{E_1}+\abs{E_2}}.\label{eq:Per 66}
\end{align}
Given $\abs{E_1}$ and $\abs{E_2}$, the configuration is uniquely determined 
up to isometries.
\end{theorem}
\begin{proof}
The $(6;6)$ configuration comprises two adjacent hexagons $E_1$, $E_2$ 
with all angles equal to 120 degrees and all straight edges.
Let $a$, $b$, $s$, be the side lengths of $E_1$ and $a$, $b$, $t$
be the sides lengths of $E_2$ where $a$ and $b$ are the length of the 
shared edges.
A straightforward computation gives:
\begin{align*}
    \Per(\T) &= s+t+2a+2b\\
    \abs{E_1} &= \frac{\sqrt 3} 2 (ab+sa+sb)\\
    \abs{E_2} &= \frac{\sqrt 3} 2 (ab+ta+tb). 
\end{align*}
Using Lagrange multipliers it is easy to check that if $\Per(\T)$ is 
minimal with respect to small variations of $a,b,s,t>0$ with 
$\abs{E_1}$ and $\abs{E_2}$ fixed, then $a=b$ and $s+t=a+b$.
Moreover, we get
\begin{align*}
    \Per(\T) &= 6 a\\
   \abs{E_1} + \abs{E_2} &= 3\sqrt{3}a^2
\end{align*}
whence $\Per(\T) = 2\sqrt[4]{3}\sqrt{\abs{D}}$ proving \eqref{eq:Per 66}. In particular, we also obtain using the above formulas
\[  
    \frac{\abs{E_1}}{\abs{D}} > \frac{\frac{\sqrt{3}}{2}a^2}{3\sqrt 3a^2} = \frac 16,
\]
that clearly implies \eqref{eq:vol 66}. 
Notice that if we change $t$ and $s$ keeping $a=b$ and $t+s=a+b$ fixed,
the lattice does not change and it is hence equal to the lattice 
of the case $a=b=t=s$ which has two adjacent regular hexagons as 
a fundamental domain.
Finally, uniqueness up to isometry is clear.
\end{proof}
    
\begin{theorem}[optimal lattice for (3;9) configuration]\label{th:computation 39}
    Let $\T$ be a non-degenerate, isoperimetric periodic $2$-tiling of the plane.
    If $\T$ is a $(3;9)$ configuration,
    then the lattice has the same periodicity of a honeycomb, i.e.,
    $E_1$ is a Reauleaux triangle while $E_2$ 
    is obtained by removing three circular arcs from every other vertex 
    of a regular hexagon.     
    Moreover, we have
    \begin{gather}
        0 < \frac{\abs{E_1}}{\abs{E_1} + \abs{E_2}} < \frac{\pi}{\sqrt 3}-1,\label{eq:admissible 39}\\
        \Per(\mathbf T) = \sqrt{2}\sqrt{\pi-\sqrt 3} \sqrt{\abs{E_1}}
        + \sqrt[4]{12}\sqrt{\abs{E_1}+\abs{E_2}}.\label{eq:Per 39}
    \end{gather}
Given $\abs{E_1}$ and $\abs{E_2}$, the configuration is uniquely determined 
up to isometries.
\end{theorem}

\begin{proof}
    Let $E$ be the hexagon obtained by prolonging the six flat edges of $E_2$.
    This hexagon generates a periodic $1$-tiling $\T'$ with the same 
    lattice as $\T$.
    Notice, moreover, that $\Per(\T')$ is obtained from $\Per(\T)$
    by subtracting the perimeter of the triangular tile $E_1$ 
    and adding the length of a tripod on the vertices of $E_1$.
    At this point, we can repeat the proof of Theorem~\ref{th:honeycomb}
    because the length $\Per(\T)$ is given by the formula~\eqref{eq:hexagon}
    with an additional term which only depends on $\abs{E_1}$ and 
    has no variation with respect to variations of the lattice.  

    Hence we find that $E$ is a regular hexagon and $\T$ and $\T'$ 
    have the same lattice. Moreover, if $r>0$ is the radius of curvature 
    of the curved edges of $E_1$, and $\ell>0$ is the side of the hexagon $E$,
    using Lemma~\ref{lm:reauleaux} we have
    \begin{align*}
        \Per(\T) &= \pi r - \sqrt 3 r + 3\ell\\
        \abs{E_1} &= \frac{\pi - \sqrt 3} 2 r^2\\
        \abs{E_1} + \abs{E_2} &= \abs{E} = \frac{3\sqrt 3}2 \ell^2
    \end{align*}
    from which we deduce \eqref{eq:Per 39}.
    Condition~\eqref{eq:admissible 39} follows by requiring 
    the Reauleaux triangle to not touch the vertices of the hexagon
    which is $\frac{r}{\sqrt 3} < \ell$.
    Uniqueness follows from Theorem~\ref{th:different_pressures}
    since the lattice is a regular hexagon.
\end{proof}

\begin{figure}
\begin{tikzpicture}[scale=0.7]
    \draw[->](-4,0) -- (6,0);
    \draw[->](0,-3) -- (0,4);
    \fill (0,0) circle(1.8pt);
    \node at  (0,-0.3) [left]{$O$};

    
    \draw[thick](3,-1) arc (-10:10:5.75) node[right] {$A=(a,b)$} -- (3,1) arc (70:110:8.76) -- (-3,1) arc (170:190:5.75) -- (-3,-1) arc (-110:-70:8.76);
    \fill (3,1) circle(1.8pt);

    \draw(0,0) -- (6,4);
    \draw node at (6,4) [right] {$B=(u,v)$};
    \fill (6,4) circle(1pt);
    \draw(2,0) -- (6,4);
    \node at (2.2,0.2) [right]{$\alpha$};
    \fill (2,0) circle(1.8pt);
    \node at (1.9,0) [below]{$P$};

    \draw(2,0) -- (5,-3);
    \draw(0,0) -- (5,-3);
    \draw node at (5,-3) [right] {$B'=(u',v')$};
    \fill (5,-3) circle(1.8pt);
\end{tikzpicture}
    \caption{Notation used in configuration (4;8).}
    \label{fig:conti}
\end{figure}
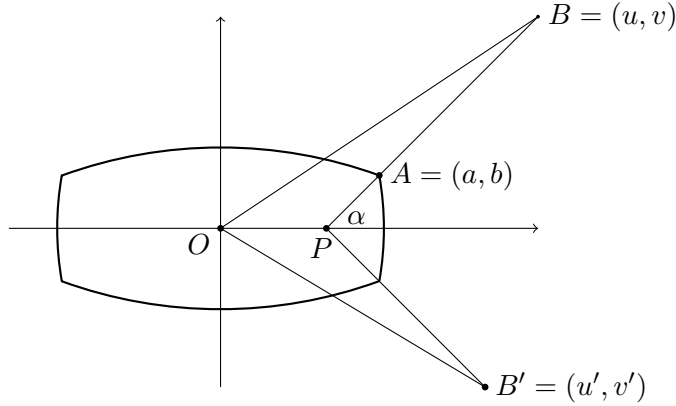
\begin{theorem}[optimal lattice for $(4;8)$ configuration]\label{th:computation 48}
Let $\T$ be a non-degenerate, isoperimetric periodic $2$-tiling of $\R^2$.
If $\T$ is a $(4;8)$ configuration, 
then the lattice admits a fundamental domain $D$ which is a square 
i.e., the lattice $G$ has two generators $g,g'\in G$ with equal length
the curvilinear rectangle, $E_1$ is a curvilinear square and the chipped parallelogram $E_2$ is a chipped square.

Moreover, we have 
\begin{gather}
    0<\frac{\abs{E_1}}{\abs{E_1}+\abs{E_2}} < \frac{(2+\sqrt 3)\left( 1-\sqrt 3 + \frac{\pi}{3}\right)}{8},\label{eq:vol 48}\\
    \Per(\T) = 2\sqrt{\frac \pi 3 + 1 - \sqrt 3} \sqrt{\abs{E_1}} + 2\sqrt{\abs{E_1}+\abs{E_2}}.\label{eq:Per 48}
\end{gather}
Given $\abs{E_1}$ and $\abs{E_2}$, the configuration is uniquely determined 
up to isometries.
\end{theorem}
\begin{proof}
Consider a system of coordinates, as in Figure~\ref{fig:conti}, with the origin in the center of the curvilinear rectangle $E_1$,
axes parallel to the axes of $E_1$ and such that the first coordinate is along 
the longest axis. 
In these coordinates the vertices of $E_1$ are $(\pm a,\pm b)$
for some $a>b>0$.
The coordinates of the midpoints of the two flat edges emitted 
from the vertices $(a,\pm b)$ are $(u,v)$ and $(u',v')$ for some $u,v,u',-v'>0$
so that $g=2(u,v)$, $g'=2(u',v')$ are two generators of the lattice $G$.
Let $0<\theta\le \frac{\pi}{12}$ 
be the angle between the arc and the chord of the edge with vertices $(a,\pm b)$, see Figure~\ref{fig:rectangle}.
Clearly, by the 120-degrees condition, 
the lines containing the flat edges emitted from the vertices $(a,\pm  b)$  
define the same angle $\alpha$ with the $x$-axis, and an easy computation gives 
$\alpha=\frac \pi 6 + \theta$.
Let $r$ be the radius of curvature of the region $E_1$.
Then:
\begin{align*}
    a &= r \sin (\pi/6-\theta) &
    b &= r \sin \theta\\
    u &= a + \ell\cos\alpha &
    v &= b + \ell\sin\alpha\\
    u' &= a + \ell'\cos\alpha &
    v' &= -b - \ell'\sin\alpha\\    
\end{align*}
where $\ell$ is the distance from $A=(a,b)$ to $B=(u,v)$ and $l'$ is the distance from $(a,-b)$ to $(u',v')$. Moreover
\begin{align*}
    \abs{D} &= 4 \det \begin{pmatrix} u' & u \\ v' & v \end{pmatrix}\\
    &= 4 ((a+\ell\cos \alpha)(b+\ell'\sin \alpha)+(a+\ell'\cos \alpha)(b+\ell\sin \alpha))\\
    &= 8 ab + 4 (a\sin \alpha +b\cos\alpha)(\ell+\ell') + 8\ell \ell'\sin\alpha\cos\alpha\\
    &\eqcolon A(\ell,\ell')
\end{align*}
while
\[
    \Per(\T) = \Per(E_1) + 2(\ell + \ell')
    \eqcolon P(\ell,\ell')
\]

\emph{Step 1.} We are going to prove that $\ell=\ell'$ which means 
that the fundamental domain $D$ is a rhombus.

Suppose by contradiction that $\ell \neq \ell'$.
If we take $\bar \ell = \frac{\ell + \ell'}{2}$
we have $P(\bar \ell,\bar \ell) = P(\ell,\ell')$ and 
$A(\bar \ell, \bar \ell) > A(\ell,\ell')$
because $\bar \ell^2 > \ell \ell'$.
Hence there exists $\lambda<1$ such that $A(\lambda \bar \ell,\lambda \bar \ell)=A(\ell,\ell')$ 
while $P(\lambda \bar \ell,\lambda \bar \ell) < P(\ell,\ell')$. 
In this process $\abs{E_1}$ is kept fixed, hence we have a contradiction 
with the minimality of $\Per(\T)$ with respect to a change of the lattice.
The condition $\ell=\ell'$ means that $D$ is a rhombus, as stated.

\emph{Step 2.} We are going to prove that $a=b$ which means that $E_1$ is a 
curvilinear square, $E_2$ is a chipped square and the fundamental domain $D$ 
is a square. 
The configuration is determined by the parameters $a,b,u,v$ where $(a,b)$
determines the shape of $E_1$ and $(u,v)$ determines the lattice $G$.
There is a compatibility condition between $a,b,u,v$ which is given by the
fact that the segment from $(a,b)$ to $(u,v)$ must define equal angles 
of 120 degrees with the two circular edges of $E_1$ joining in $(a,b)$.
By Lemma \ref{lm:quadrangular}
we have:
\begin{align*}
    \abs{D} &= 8 uv\\
    \abs{E_1} &= r^2\enclose{1-\sqrt 3 + \frac \pi 3 - 2 \sin^2 (\theta - \frac{\pi}{12})}\\
    \Per(\T) &= \frac 2 3 \pi r + 4\sqrt{(u-a)^2+(v-b)^2}\\
    r &= r(a,b) = 2\sqrt{a^2+\sqrt 3 ab + b^2}.
\end{align*}
Let $m=\min\ENCLOSE{\abs{E_1},\abs{D}-\abs{E_1}}$
and let $E_1'$ be the curvilinear square with angles of 120 degrees and vertices 
in the points $(\pm c,\pm c)$
with $c>0$ appropriately chosen so that $\abs{E_1'}=m$ i.e., 
by Lemma~\ref{lm:quadrangular},
\[
    r^2(c,c)\enclose{1-\sqrt 3 + \frac \pi 3} 
    = 4c^2 (2+\sqrt 3) \cdot \enclose{1-\sqrt 3 + \frac \pi 3}
    = m.
\]
Take $w=\sqrt{uv}$ and consider the group $G'$ generated by $2(w,w)$ and $2(w,-w)$
which has a fundamental domain $D'$ which is a square with vertices 
$(0,0), 2(w,w), 2(w,-w), (4w,0)$ and the same area as $D$:
$\abs{D'} = 8w^2 = 8uv$.
Notice that $\abs{E_1'}$ strictly contains a square of side $2c$, 
hence $\abs{E_1'}>4c^2$ and since $\abs{E_1'}=m$ 
is smaller than half the area of $D$ 
we obtain that $c<w$.
So, the translations $E_1'+g'$ with $g'\in G'$ are pairwise disjoint and we can define 
a the tile 
\[
    E_2'=D'\setminus \bigcup_{g'\in G'}(E_1'+g') 
\]
so that $E_1'$ and $E_2'$ generate a tiling $\T'$ of type $(4;8)$
which is a competitor to the tiling $\T$ because 
either $\abs{E_1'}=\abs{E_1}$ and $\abs{E_2'}=\abs{E_2}$ or
$\abs{E_1'}=\abs{E_2}$ and $\abs{E_2'}=\abs{E_1}$.

We are going to show that if $a>b$ then $\Per(\T')<\Per(\T)$.

Notice that the level curves of the function $r(x,y)$ are ellipses centred in the origin 
and with axes $y=\pm x$.
Take $d>0$ such that $r(d,d) = r(a,b)$ and notice that 
\[
    r^2(c,c)\enclose{1-\sqrt 3 + \frac \pi 3}
    = \abs{E_1'} 
    \le \abs{E_1}
    < r^2(a,b) \enclose{1-\sqrt 3 + \frac \pi 3} 
\]
hence $r(c,c)<r(a,b)=r(d,d)$ and $c<d$.
Define 
\[
  P(x,z) 
    = \frac 2 3 \pi r(x,x) + 4\sqrt 2 (z-x)
    = \enclose{\frac 4 3 \pi \sqrt{2+\sqrt 3}-4\sqrt 2} x + 4\sqrt 2 z
\]
and check that this function is increasing in $x$ so that 
$\Per(\T') = P(c,w) < P(d,w)$.
Since $(d,d)$ and $(a,b)$ are two points on the same level curve of 
$r(x,y)$ which is an ellipse $E$ while $(w,w)$ and $(u,v)$ 
are on the same level curve of the function $xy$ which is a hyperbola $H$,
and clearly the smaller signed distance between $E$ and $H$ is obtained
by taking the points on the axis $y=x$, 
we have that
\[
    \sqrt{(u-a)^2+(v-b)^2} 
    \ge \sqrt 2 \abs{w-d} \ge w-d.
\]
This finally gives
\[
    \Per(\T) 
        = \frac 2 3 \pi r(a,b) + 4 \sqrt{(u-a)^2+(v-b)^2}
        \ge P(d,w) > P(c,w) = \Per(\T')
\]
as claimed.

So we have $a=b$ and $u=v$. Hence
\begin{align*}
    \Per(\T) &= P(a,u) = \enclose{\frac 4 3 \pi \sqrt{2+\sqrt 3} - 4 \sqrt 2} \cdot a + 4\sqrt 2 u\\
    \abs{E_1} &= a^2 (2+\sqrt 3)\enclose{1-\sqrt 3 + \frac \pi 3}\\
    \abs{D} &= 8u^2
\end{align*}
Hence
(the algebraic simplification requires some work, but can easily be checked numerically)
\begin{align*}
    \Per(\T) 
        &= \frac{2 \pi - 6\sqrt{4-2\sqrt 3}}{\sqrt{9 + 3\pi - 9\sqrt 3}} \sqrt{\abs{E_1}} 
        + 2\sqrt{\abs{D}}\\
        &= 2\sqrt{\frac \pi 3 + 1 - \sqrt 3} \sqrt{\abs{E_1}} + 2\sqrt{\abs{D}}
\end{align*}
proving \eqref{eq:Per 48}. 
Furthermore, \eqref{eq:vol 48} follows by the above formulas for $\abs{E_1},\abs{D}$ and the observation that when $u=a$ the configuration (4;8) is not anymore admissible.
Finally, uniqueness follows from Theorem~\ref{th:different_pressures}.
\end{proof}
We can now finally prove our main result around the isoperimetric profile function. The following directly implies our second main result Theorem \ref{thm:main 2}.
\begin{theorem}[isoperimetric profile]\label{thm:isoperimetric_profile}
Let $\T$ be an isoperimetric $2$-tiling of the plane
and denote by $E_1$ and $E_2$ two generators of $\T$. 
Let $x=\frac{\abs{E_1}}{\abs{E_1}+\abs{E_2}}$ so that  $\mathcal I(x) = \frac{\Per(\mathbf T)}{\sqrt{\abs{E_1}+\abs{E_2}}}$. 
Define:
\begin{align*}
m_1 &\coloneq \sqrt 2\sqrt{\pi -\sqrt 3} & q_1 &\coloneq \sqrt[4]{12} \\
m_2 &\coloneq 2\sqrt{\frac{\pi}{3}+1-\sqrt 3} & q_2 &\coloneq 2\\
 & & q_3 &\coloneq 2\sqrt[4]{3}
\end{align*}
and 
\[
x_1 \coloneq \frac{(q_2-q_1)^2}{(m_1-m_2)^2} \approx 0.062
\qquad 
x_2 \coloneq \frac{(q_3-q_2)^2}{m_2^2} \approx 0.317.
\]
Then we have (see Figure~\ref{fig:profile}):
\begin{enumerate} 
    \item[(0)] if $x=0$ then $\mathbf T$ is the Honeycomb tiling and $\mathcal I(0) = f(0)=q_1$;
    \item if $0 <x < x_1$ then $\T$ is a $(3;9)$ configuration and $\mathcal I(x) = m_1\sqrt x + q_1$;
    \item if $x=x_1$ then $\T$ can either be a $(3;9)$ configuration or a $(4;8)$ configuration and $\mathcal I(x_1) = m_1\sqrt x + q_1 = m_2\sqrt x + q_2$;
    \item if $x_1 < x < x_2$ then $\T$ is a $(4;8)$ configuration and $\mathcal I(x) = m_2\sqrt x + q_2$;
    \item if $x=x_2$ then $\T$ can either be a $(4;8)$ configuration or a $(6;6)$ configuration and $\mathcal I(x) = m_2\sqrt x + q_2 = q_3$;
    \item if $x_2 < x \le \frac 12$ then $\T$ is a $(6;6)$ configuration and $\mathcal I(x) = q_3$;
    \item if $x \in [\frac 12,1)$ then $\mathcal I(x) = \mathcal I(1-x)$.
\end{enumerate}
Finally, if $x \notin \{ x_1,x_2,1-x_2,1-x_1\}$ then 
isoperimetric $2$-tilings are uniquely determined up to isometries.
\end{theorem}
\begin{proof}
    First, we notice that (0),(6) are straightforward. Indeed, (0) is directly given by Theorem \ref{th:honeycomb} while (6) is clear by the very definition of $\mathcal I(x)$.  
    Observe also that $x_1,x_2$ can be readily computed by solving
    $m_1\sqrt{x} + q_1 = m_2 \sqrt{x} + q_2$ and $m_2\sqrt{x}+q_2=q_3$.
    We can infer that:
    \begin{enumerate}
        \item[(1)] holds by Theorem \ref{th:computation 39} since $\mathcal I(x) = m_1\sqrt x + q_1$ for $x \in (0,x_1]$ (observe $x_1 \approx 0.062< \frac{\pi}{\sqrt 3}-1 $); 
        \item[(3)] holds by Theorem \ref{th:computation 48} since $\mathcal I(x)=m_2\sqrt x + q_2$ for $x \in [x_1,x_2]$. In particular, (2) also holds;
        \item[(5)] holds by Theorem \ref{th:computation 66} since $\mathcal I(x)=q_2$ if $x \in \left[x_2,\frac 12\right]$. In particular, (4) also holds.
    \end{enumerate}
    Finally, uniqueness follows by the uniqueness stated in Theorems~\ref{th:computation 39}, \ref{th:computation 48} and \ref{th:computation 66}.
\end{proof}
\section{Open problems}\label{sec:open problems}
We conclude this note by raising some natural open problems related to our investigation.
\subsection{Explicit profile $\mathcal I_G$}
Our main result in Theorem \ref{thm:main 2} provides an explicit description of the profile $\mathcal I(\cdot)$. 
Recall that the latter is obtained by minimizing the isoperimetric profile function $\mathcal I_G(\cdot)$ among all lattices $G$. 
However, in Theorem~\ref{th:main 1}, we also achieve classification results for isoperimetric tilings relative to a lattice $G$. 
In particular, we expect a relation between the area $x \in (0,\frac 12)$ and the lattice $G$ describing when the configurations (3;9), (4;8), (6;6) are isoperimetric. 
In this fixed lattice scenario and especially for the configuration (4;8), the computations carried in the proof of Theorem \ref{th:computation 48} become quite implicit, suggesting that this relation could be hard to achieve. It is then natural to raise the following problem:
\begin{open}\label{prob:G exists}
    Are there lattices $G$ in $\R^2$ such that no periodic $2$-tilings isoperimetric relative to $G$ are of type (4;8)?  
\end{open}
Understanding the above problem seems to be a necessary step to study the isoperimetric profile function relative to a lattice. The challenging part is to understand for which area a configuration (4;8) is admissible depending on the fixed lattice. We raise it as a second problem.
\begin{open}\label{prob:IG explicit}
    Given any lattice $G$ in $\R^2$, it is possible to give an explicit description of $\mathcal I_G(x)$ for $x\in [0,1]$?
\end{open}
\subsection{Isoperimetric $N$-tilings}
Given $N\in\N$, it would be interesting to understand to what extent our analysis generalizes to the case of isoperimetric $N$-tilings in $\R^2$. In some cases, there are natural candidates for minimizers as we comment here.

By \cite[Theorem 5.2]{NobiliNovaga24}, we know that a periodic tiling composed by $N$ adjacent hexagons with almost the same amount of area is isoperimetric. Again, we stress that this was achieved by a qualitative argument and we do not know how close the areas should be for this to happen. 

Furthermore, if $N=3$ with two cells having very small areas compared to the third one, a natural candidate is obtained by adding to the configuration $(3;9)$ a second Reuleaux triangle on an empty vertex of the regular hexagon. 

Finally, a combination of these situations might also occur when $N$ cells have almost the same area while other $2N$ cells or fewer have smaller areas close to zero. 
In this case, a natural candidate is the tiling composed of $N$ adjacent regular Hexagons and $2N$ Reuleaux triangles on the vertices.
\begin{figure}
\includegraphics[width=0.5\textwidth]{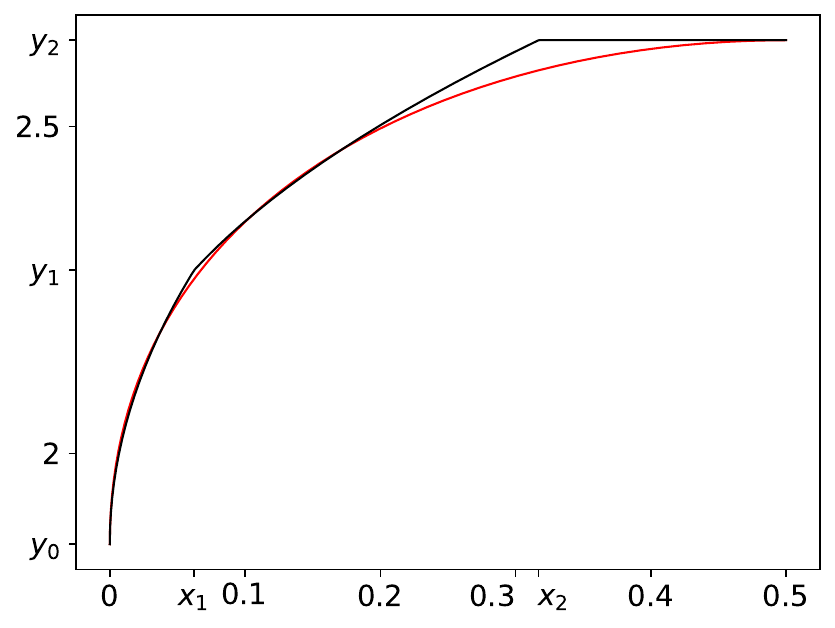}%
\includegraphics[width=0.5\textwidth]{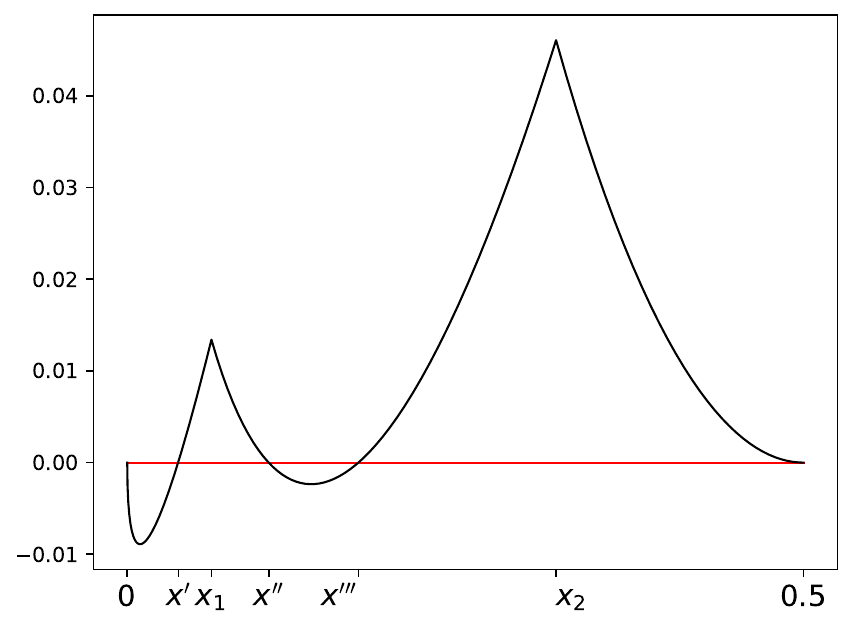}
    \caption{Left: the isoperimetric profile $\mathcal I(x)$ of the periodic $2$-tiling, compared with the asymptotic profile $\mathcal J(x)$, see~\eqref{eq:kelvin}, of a non-periodic tiling composed of hexagons with the same areas.
Right: the difference $\mathcal I(x) - \mathcal J(x)$, highlighting the intervals where $\mathcal I(x) < \mathcal J(x)$.}
\label{fig:comparison} 
\end{figure}
\subsection{Kelvin problem with unequal cells}
Recall that the Kelvin problem asks for a partition of the Euclidean space into cells of equal volume and with minimal interface area. In dimension $d=2$, the Honeycomb tiling was conjectured to be minimal and this has been affirmatively proved by Hales in \cite{Hales01}. Minimality, for instance, could be thought locally under compact perturbations as in the following definition.
\begin{definition}\label{def:constrained local minimizer}
Let $(E_i)_{i \in I}$ be a partition up to negligible sets of $\R^d$ with $|E_i|<\infty$ for all $i \in I$. 
We say that $(E_i)$ is a volume constrained local minimizer for the perimeter,
shortly local minimizer,
provided for any ball $B_r\subset\R^d$ of radius $r>0$ with 
\[
    \sum_{i\in I}\Per(E_i,B_r)<\infty,
\]
and for all other partitions $(F_i)_{i \in I}$ satisfying $E_i\triangle F_i\Subset B_r$ and $|E_i|=|F_i|$ for all $i \in I$, it holds
\[
    \sum_{i\in I}\Per(E_i,B_r) \le \sum_{i\in I} \Per(F_i,B_r).
\]
\end{definition}
In the Kelvin problem, the volumes of the partitions $(E_i),(F_i)$ are kept fixed and, for simplicity, equal to one. However, given the results of these notes, a natural problem would be to investigate  the Kelvin problem 
for partitions which have two different assigned volumes.

At least in the planar case, it is reasonable to ask if the partition $(E_i)$ of $\R^2$ 
induced by an isoperimetric periodic tiling satisfying $\abs{E_{2i}}=x$,
$\abs{E_{2i+1}}=1-x$
are local minimizers for the perimeter.
It is easy to understand that a possible competitor $(F_i^r)$ on a ball $B_r$ 
is obtained by subdividing the ball into two regions 
of measure proportional to $x$ and $1-x$  and fill each area 
with a honeycomb made of hexagons of area $x$ in the first region 
and area $1-x$ on the second one. One would have
\begin{equation}\label{eq:kelvin}
  \lim_{r\to+\infty}\frac{\sum_{i\in I}\frac 12 \Per(F_i^r,B_r)}{\pi r^2} = \sqrt[4]{12}\cdot(\sqrt{x}+\sqrt{1-x}) \eqcolon \mathcal J(x).
\end{equation}
In Figure~\ref{fig:comparison} we have superimposed the plot of $\mathcal J(x)$ with $\mathcal I(x)$.
One can notice that $\mathcal I(x)=\mathcal J(x)$ only for 
$x\in\{0, x', x'', x''',1\}$
with $x'\approx 0.038$, $x''\approx 0.105$ and $x'''\approx 0.171$
and $\mathcal I(x) \le \mathcal J(x)$ for $x\in [0,x']\cup[x'',x''']$.
Hence for all other values of $x$, 
the isoperimetric periodic tiling is not a local minimizer.
This might suggest that minimizers exist only for some choices of the assigned areas, for example when we have 
the configurations $(3;9)$ and $(4;8)$. In particular, a natural question is understanding if our configurations are local minimizers in the range where they are better than the tiling with regular hexagons.

\appendix
\section{Useful geometric computations}
In this appendix, we collect useful results and key computations around planar curvilinear polygons. These will be needed to analyze our candidate minimizers.

\begin{lemma}[turning angle]
    \label{lm:turning}
    Let $C$ be a simply connected curvilinear polygon delimited by edges which are either circular arcs 
    or straight segments joining in the vertices with equal angles of 120 degrees.
    Let $\alpha_1,\dots,\alpha_n$ be the angle spanned by each edge with $\alpha_k>0$ if the edge is convex,
    $\alpha_k<0$ if the edge is concave and $\alpha_k=0$ if the edge is flat.
    Then 
    \begin{equation}\label{eq:turning}
      \sum_{k=1}^n \alpha_k = (6-n)\cdot 60^\circ.
    \end{equation}
    \end{lemma}
    \begin{proof}
    Start from a vertex of $C$ and follow the normal vector to the edges around $C$ in 
    the counter-clockwise direction.
    On each vertex the normal vector rotates counterclockwise of an angle of 60 degrees.
    Along each edge, the normal vector rotates counterclockwise of the angle $\alpha_k$ 
    corresponding to the edges. 
    After making a complete round trip the vector has made a rotation of 360 degrees. 
    Hence we have 
    \[
     \sum_{k=1}^n (\alpha_k + 60^\circ) = 360^\circ
    \]
    and \eqref{eq:turning} follows.
    \end{proof}

\begin{lemma}[area and length of Reuleaux triangle]
\label{lm:reauleaux}
    Let $E\subset \R^2$ be a Reuleaux triangle i.e., a curvilinear polygon composed of three arcs joining with equal angles of 120 degrees in the vertices. 
    All the edges have the same radius of curvature $r>0$
    and we have
    \begin{align}
        \abs{E} &= \frac{\pi-\sqrt 3}{2} r^2, \qquad
        r = \frac{\sqrt{2\abs{E}}}{\sqrt{\pi -\sqrt 3}},\\
        \Per(E) &= \pi r.
    \end{align}
    Moreover, the distance of each vertex  to the center of the triangle is given by $r/\sqrt 3$. 
\end{lemma}
\begin{proof}
These are straightforward computations.
\end{proof}

\begin{lemma}[length of a Steiner tripod]
Let $A=(0,0)$, $B=(u,0)$, and $C=(w,v)$ be three points in $\R^2$
and let $S$ be a point such that the three 
segment $AS$, $BS$, $CS$ define equal angles of 120 degrees in $S$.
Then 
\[
    \abs{A-S}+\abs{B-S}+\abs{C-S}
    = \sqrt{ \enclose{w-\frac u 2}^2 +\enclose{v+\frac{\sqrt{3}}{2} u}^2}.
\]
\end{lemma}
\begin{proof}
By Melzak's construction \cite{Mel61}
the length of the tripod is equal to the distance $\abs{C-M}$ 
where $M=\enclose{\frac u 2,-\frac{\sqrt{3}}{2}u}$ is the third vertex of 
the equilateral triangle with vertices in $A$ and $B$, opposite to $C$.
The result follows.
\end{proof}

\begin{figure}
\begin{tikzpicture}[scale=0.5]
    \draw[thick](3,-1) arc (-10:10:5.8) -- (3,1) arc (70:110:8.77) -- (-3,1) arc (170:190:5.77) -- (-3,-1) arc (-110:-70:8.77);
    \draw[dashed] (2.98,-0.98) -- (2.98,0.98) -- (-2.98,0.98) -- (-
    2.98,-0.98) -- cycle;

    \draw[thin](3.01,0.7) -- (3.4,1);
    \node at (3.6,1.17) {$\theta$};

    \draw[thin](2.7,1.04) -- (2.5,1.7);
    \node at (2.5,2) {$\theta'$};    
\end{tikzpicture}
    \caption{A quadrangular curvilinear polygon with circular edges.}
    \label{fig:rectangle}
\end{figure}
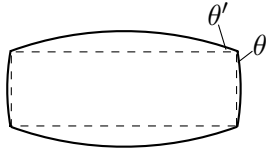
        
In the next lemma, we find the compatibility condition for a quadrangular curvilinear polygon to exist, for a given volume and a given radius of the circular edges.

\begin{lemma}[quadrangular curvilinear polygon]
    \label{lm:quadrangular}
    Let $E$ be a curvilinear quadrilateral composed of four convex circular arcs of equal 
    curvature, meeting with angles of 120 degrees
    (see Figure~\ref{fig:rectangle}).
    Then $E$ has the symmetries of a rectangle: the axis of each edge is a symmetry axis 
    for the whole curvilinear polygon $E$. 
    In particular, opposite edges have equal lengths.
    Moreover, if $r$ is the common curvature radius of the edges,
    $\theta$ is the angle between one of the smaller edges and its chord,
    $2a$ is the chord of the largest arc
    and $2b$ is the chord of the smaller arc,
    then
    we have
    \begin{align*}
        \theta &\in (0,\pi/12],\quad
        b = r\sin \theta, \quad 
        a = r \sin\enclose{\frac \pi 6 - \theta},\\
        \abs{E} &= r^2 \Enclose{1-\sqrt 3 + \frac \pi 3 - 2 \sin^2\enclose{\theta-\frac \pi{12}}},\\
        \Per(E) &= \frac 2 3 \pi r,\\
         r &= 2\sqrt{a^2+\sqrt 3 ab + b^2}.
    \end{align*}

    In particular, for $\theta= \frac{\pi}{12}$, the quadrangular curvilinear polygon is a curvilinear 
    square.
\end{lemma}
\begin{proof}
\emph{Step 1.}  First of all, we claim that the quadrilateral must have the symmetries of a rectangle. 
If opposite arcs of the quadrilateral have different centers, 
then the claim easily follows from \cite[Lemma 5.1]{Wichiramala} (just notice that if we fix an edge $e$ the two adjacent edges are contained in two circles which must be symmetric with respect to the axis of $e$).
So it is enough to prove that opposite arcs cannot have the same center. 
Suppose by contradiction that they have the same center. Then the other pair of arcs must have the same center as well
(simply because they have the same curvature radius, and the curvilinear quadrilateral is assumed to be convex). 
But then the region $E$ would be a disk, contradicting the assumption that the vertices have 120 degrees. 

\emph{Step 2.} 
First, notice that $\theta$ must satisfy $\theta \in \left(0,\frac \pi {12}\right)$, since circular edges meet at a common vertex with an angle equal to $\frac 2 3 \pi$.

The fact that $\Per(E)=\frac 2 3 \pi r$ easily follows by Lemma \ref{lm:turning}, noticing that the turning angle is precisely $\frac 2 3 \pi$.  
Equivalently, this follows from the following computation 
$\Per(E)= 4\theta r + 4\enclose{\frac \pi 6 - \theta}r = \frac 2 3 \pi r$. 
The length of the chords (edges of the corresponding straight rectangle)
are respectively
$2r\sin\enclose{\theta - \frac \pi 6}$
and $2r\sin\theta$.
Hence the area of the rectangle with straight edges is:
\[
    V = 4r^2\sin\theta\sin\enclose{\frac \pi 6-\theta}
\]
while the areas of the two different circle sectors between the arcs and the chords are
$v(\theta)$ and $v\enclose{\frac \pi 6 - \theta}$ with
\[
    v(\theta) = \theta r^2 - \frac {r^2} 2 \sin(2\theta).
\]
hence
\begin{equation}
\begin{aligned}
    \abs{E} &= V + 2v(\theta) + 2 v\enclose{\frac \pi 6 - \theta} \\
    & = r^2 \Enclose{4\sin \theta \sin\enclose{\frac \pi 6-\theta}
    + 2 \theta - \sin(2\theta) +  \frac \pi 3 - 2\theta - \sin\enclose{\frac \pi 3 - 2\theta}}.
\end{aligned}
\label{eq:compatibility theta polygon}
\end{equation}
By trigonometric formulas, the sought identity for $\abs{E}$  follows. Finally, if $\theta = \frac \pi {12}$, then the inner rectangle is a square. 

\emph{Step 3.} Let us compute the curvature radius $r$ and conclude the proof.
We claim that the centre $P$ of the arc passing through the points $A=(a,b)$ 
and $(a,-b)$ has coordinates $P=(-a-\sqrt 3 b,0)$, while 
the center $Q$ of the arc passing through $(a,b)$ and $(a,-b)$
has coordinates $Q=(0,-b-\sqrt 3 a)$.
In fact it is a nice exercise in planar geometry to prove that 
if we construct two equilateral triangles on the exterior of two adjacent 
edges of a straight rectangle, the two constructed vertices ($P$ and $Q$ in our case) together with 
the opposite vertex of the rectangle ($A=(a,b)$ in our case), 
form an equilateral triangle.
This means that the arcs centred in these two points and passing through the point $A$ have the same radius and define an angle of 120 degrees in $A$.
It is not difficult to show that the points $P$, $Q$ are uniquely defined 
by these properties, obtaining the claim.
Hence we have 
\[
    r^2 = \abs{A-P}^2 = (a + a + \sqrt 3 b)^2 + (b-0)^2 
    = 4 a^2 + 3b^2 +4\sqrt 3 ab + b^2 
    = 4 a^2 + 4\sqrt 3 a b + 4 b^2.
\]
\end{proof}
\begin{lemma}[planar periodic graphs]\label{lm:eulero}
    Let $G$ be a simple planar graph whose set of vertices is $\Z^2 \subset \R^2$ and whose set 
    of edges $E$ is invariant under the action of $\Z^2$.
    Then each vertex of $G$ has order (number of edges) $d\le 6$.
\end{lemma}
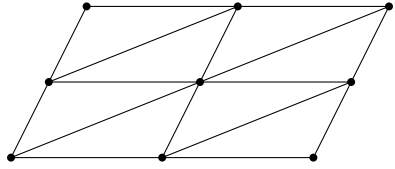
\begin{figure}
\begin{tikzpicture}
    \fill (0,0) circle(1.5pt);
    \fill (2,0) circle(1.5pt);
    \fill (4,0) circle(1.5pt);
    \fill (0.5,1) circle(1.5pt);
    \fill (2.5,1) circle(1.5pt);
    \fill (4.5,1) circle(1.5pt);
    \fill (1,2) circle(1.5pt);
    \fill (3,2) circle(1.5pt);
    \fill (5,2) circle(1.5pt);

    \draw (0,0) -- (4,0) -- (5,2) -- (1,2) -- (0,0);
    \draw (0.5,1) -- (4.5,1);
    \draw (2,0) -- (3,2);
    \draw (0,0) -- (5,2);
    \draw (2,0) -- (4.5,1);
    \draw (0.5,1) -- (3,2);
\end{tikzpicture}
\caption{A planar periodic graph can have vertices with a maximum order of 6.}
\end{figure}
\begin{proof}
Consider the subgraph $G_R$ of $G$ composed by all facets of $G$ contained in a large ball 
of radius $R$.
The graph $G_R$ has $V_R$ vertices, $E_R$ edges, and $F_R$ facets.
Being planar, Euler formula gives: $V_R-E_R+F_R=2$.
Each facet of $G_R$ has at least $3$ edges, and each edge is shared by at most two facets:
$3 F_R \le 2 E_R$. 
Each vertex of $G$ has exactly $d$ edges and each edge joins exactly two vertices
so that $d V_R = 2 E_R+o(E_R)$ where $o(E_R)$ accounts for the edges of $G$ which are joining 
a vertex of $G_R$ with a vertex of $G \setminus G_R$. 
So, Euler formula gives
\[
2 
= V_R-E_R+F_R 
\le \frac 2 d E_R + o(E_R) - E_R + \frac 2 3 E_R
= \enclose{\frac 2 d - \frac 1 3} E_R + o(E_R). 
\]
Letting $R\to +\infty$ we must have $\frac 2 d - \frac 1 3 \ge 0$, which means $d\le 6$.
\end{proof}

\medskip

\noindent\textbf{Acknowledgments}. 
All authors are members of INDAM-GNAMPA. 
The authors acknowledge support from the MIUR Excellence Department Project awarded to the Department of Mathematics, University of Pisa, CUP I57G22000700001. 
M.N. acknowledges partial support from the PRIN 2022 project 2022E9CF89, and the PRIN 2022 PNRR project P2022WJW9H. 
E.P. acknowledges support from the PRIN 2022 project 2022PJ9EFL, and from the project PRA 2022 14 GeoDom (Università di Pisa). The authors thank the anonymous referee for valuable comments that significantly improved the manuscript.

\bibliographystyle{siam}

\end{document}